\documentclass[a4paper,reqno]{amsart}
\usepackage{mathtools}
\usepackage[inline]{enumitem}
\usepackage{amsmath}
\usepackage{xcolor}
\usepackage[citebordercolor=olive,urlbordercolor=gray]{hyperref}
\usepackage{amssymb}
\usepackage{amsthm}
\usepackage{mathrsfs}
\usepackage{amsfonts}
\usepackage{graphicx}
\usepackage{pdfpages}
\usepackage{lipsum}
\usepackage[all,2cell]{xy}\UseAllTwocells \SilentMatrices
\xyoption{v2}
\SelectTips{cm}{10}

\usepackage[british]{babel}

\newcommand{\ul}[1]{\underline{\mathcal{#1}}}
\newcommand{\ca}{\mathcal}
\newcommand{\Hom}{\ensuremath{\mathrm{Hom}}}
\newcommand{\Comod}{\ensuremath{\mathbf{Comod}}}
\newcommand{\Mod}{\ensuremath{\mathbf{Mod}}}
\newcommand{\Alg}{\ensuremath{\mathbf{Alg}}}
\newcommand{\Coalg}{\ensuremath{\mathbf{Coalg}}}
\newcommand{\Mon}{\ensuremath{\mathbf{Mon}}}
\newcommand{\Comon}{\ensuremath{\mathbf{Comon}}}
\newcommand{\Bimon}{\ensuremath{\mathbf{Bimon}}}

\newcommand{\Vect}{\ensuremath{\mathbf{Vect}}}

\newcommand{\op}{\mathrm{op}}

\newcommand{\ev}{\ensuremath{\mathrm{ev}}}
\newcommand{\coev}{\ensuremath{\mathrm{coev}}}
\newcommand{\avb}{\ensuremath{\prescript{A}{}}{\mathcal{V}}_B}

\theoremstyle{plain}
\newtheorem{thm}{Theorem}[section]
\theoremstyle{plain}
\newtheorem{prop}[thm]{Proposition}
\theoremstyle{remark}
\newtheorem{rmk}[thm]{Remark}
\newtheorem{ex}[thm]{Example}
\theoremstyle{plain}
\newtheorem{lem}[thm]{Lemma}
\theoremstyle{plain}
\newtheorem{cor}[thm]{Corollary}
\theoremstyle{definition}
\newtheorem{df}[thm]{Definition}

\begin{document}
\title{Hopf measuring comonoids and enrichment}
\author[M.~Hyland]{Martin Hyland}
\address{Department of Pure Mathematics and Mathematical Statistics,
  University of Cambridge, Cambridge CB3 0WB, UK.}
\email{M.Hyland@dpmms.cam.ac.uk}
\author[I.~L\'opez Franco]{Ignacio L\'opez Franco}
\address{Department of Pure Mathematics and Mathematical Statistics,
  University of Cambridge, Cambridge CB3 0WB, UK.}
\email{ill20@cam.ac.uk}
\thanks{The second author gratefully acknowledges the support of a Research Fellowship of
  Gonville and Caius College, Cambridge}
\author[C.~Vasilakopoulou]{Christina Vasilakopoulou}
\address{Department of Mathematics, Massachusetts Institute of Technology,
Cambridge MA 02139, USA.}
\email{cvasilak@mit.edu}
\thanks{The third author gratefully acknowledges the financial support by Trinity
  College, Cambridge and the Department of Pure Mathematics and Mathematical
  Statistics of the University of
  Cambridge, as well as Propondis Foundation and Leventis Foundation.}
\begin{abstract}
  We study the existence of universal measuring comonoids $P(A,B)$ for a pair of
  monoids $A$, $B$ in a braided monoidal closed category, and the associated
  enrichment of the category of monoids over the monoidal category of
  comonoids. In symmetric categories, we show that if $A$ is a bimonoid and $B$
  is a commutative monoid, then $P(A,B)$ is a bimonoid; in addition, if $A$ is a
  cocommutative Hopf monoid then $P(A,B)$ always is Hopf. If $A$ is a Hopf
  monoid, not necessarily cocommutative, then $P(A,B)$ is Hopf if the
  fundamental theorem of comodules holds; to prove this we give an alternative
  description of the dualizable $P(A,B)$-comodules and use the theory of Hopf
  (co)monads. We explore the examples of universal measuring comonoids in vector
  spaces and graded spaces.
\end{abstract}

\maketitle
\section{Introduction}\label{Introduction}

The \emph{finite} or \emph{Sweedler dual} of a $k$-algebra~\cite{Sweedler} plays
a central role in the duality theory of Hopf algebras. If $A$ is an algebra over
a field $k$, its finite dual $A^\circ$ is a coalgebra with the property that
coalgebra morphisms $C\to A^\circ$ are in natural bijection with algebra morphisms $A\to
C^*$, for any coalgebra $C$. When $A$ has finite dimension, $A^\circ$ is
isomorphic to the linear dual $A^*$, but in arbitrary dimension $A^*$
may not have a natural coalgebra structure.

The classical construction of the finite dual $A^\circ$ depends on the fact that
$k$ is a field, a hypothesis that was somewhat weakened in~\cite{MR1780737}. The
existence of a coalgebra $A^\circ$ satisfying the universal property described
in the previous paragraph can be proven in great generality
(see~\cite{AdjAlgCoalg}, but also Section~\ref{existmeascoal} where a braiding
is not required); in particular, $A^\circ$ exists over any commutative ring, but
its classical description may no longer hold true.

That fact that the category of $k$-algebras admits an enrichment in the category
of $k$-coalgebras has long been part of mathematical folklore. It seems that
Gavin Wraith was the first to appreciate this fact and that he lectured on it
and related matters at the Oberwolfach meeting \emph{Universelle und
  Kategorische Algebra}, 3--10~July 1968. When Sweedler's book~\cite{Sweedler} on
Hopf algebras came out, Wraith immediately recognised that the enrichment is
given by what Sweedler called the \emph{universal measuring coalgebra} of a pair
of algebras.

In the present paper we explore the existence of a generalisation of the finite
dual construction, called \emph{universal measuring comonoids}
$P(A,B)$, for a pair of monoids $A$, $B$ in a monoidal closed category, and the
properties of this construction. The comonoid $P(A,B)$ is defined by the the
property that monoid morphisms
$A\to[C,B]$ are in natural bijection with comonoid morphisms $C\to P(A,B)$, for all
comonoids $C$; note that $A^\circ=P(A,I)$, where $I$ is the monoidal unit. We
show that, when the monoidal category has a braiding, the functor 
with mapping on objects $(A,B)\mapsto
P(A,B)$ is monoidal, so there are coherent comonoid morphisms
\begin{equation}
  \label{eq:50}
  P(A,B)\otimes P(A',B')\longrightarrow P(A\otimes A',B\otimes B')
  \quad\text{and}\quad
  I\longrightarrow P(I,I)
\end{equation}
(the latter is invertible), and when the braiding is a symmetry, $P$ is a
braided monoidal functor.

The enrichment of the opposite of the category $\Mon(\mathcal{V})^{\mathrm{op}}$
of monoids in $\mathcal{V}$ over the category $\Comon(\mathcal{V})$ of comonoids
in $\mathcal{V}$ arises from an action of the latter, viewed as a monoidal
category, on the former. We are lead to consider actions of monoidal categories
and answer the following question: what extra structure on an action of the
monoidal category $\mathcal{C}$ on $\mathcal{A}$ ensures that the associated
$\mathcal{C}$-enriched category is monoidal? This extra
structure is what we call an \emph{opmonoidal action}, and we use it to deduce that for
a symmetric $\mathcal{V}$, the category of monoids is a symmetric monoidal
$\Comon(\mathcal{V})$-category.

The classical construction of the Sweedler dual $A^\circ$ of a
$k$-algebra~\cite{Sweedler} satisfies two important properties: $A^\circ$ is a
bialgebra if $A$ is so, and $A^\circ$ is Hopf algebra if $A$ is so. We show in
complete generality that $P(A,B)$ is a bimonoid if $A$ is a bimonoid and $B$ is
a commutative monoid. We then prove that $P(A,B)$ is a Hopf monoid in two situations.
First, if $A$ is a Hopf cocommutative bimonoid and $B$ is commutative; secondly,
if $A$ is a Hopf monoid (not necessarily cocommutative) and the base symmetric
monoidal category satisfies the fundamental theorem of comodules. To prove this
last result, we provide an alternative description of the dualizable
$P(A,B)$-comodules, as dualizable objects $X$ equipped with a morphism
$A\otimes X\to X\otimes B$ that satisfies two axioms.

We now briefly outline the organisation of the article. Section~\ref{background}
collects some known facts about monoidal closed categories, monoids and
comonoids, and locally presentable categories. After recalling the connection
between actions of monoidal categories and enrichment, Section~\ref{actions}
introduces \emph{opmonoidal actions} and \emph{braided opmonoidal actions}, and
proves that they give rise to monoidal and braided monoidal enriched
categories. Section~\ref{existmeascoal} studies the existence of universal
measuring comonoids in a more general setting
than the category of $R$-modules of~\cite{AdjAlgCoalg},
namely, in locally presentable monoidal
categories. The enrichment of monoids in comonoids is recovered in
Section~\ref{enrichmonscomons}, while Section~\ref{sec:calc-meas-coalg} gives
some tools to compute universal measuring comonoids, especially
via their comodules. Section \ref{sec:monoidal-structures}
explores induced monoidal structures of dualizable comodules, and \ref{sec:univ-meas-coalg-1}
certain (co)commutativity relations for $P(A,B)$. 
In Section~\ref{sec:meas-comon-hopf}, we move to the Hopf setting
by proving that the universal measuring
comonoids of cocommutative Hopf monoids are Hopf monoids,
in the general context of a locally presentable symmetric monoidal
closed category. The case when the Hopf monoid is not necessarily cocommutative
is dealt with in Section~\ref{sec:univ-meas-comon}, which also contains some aspects
of the theory of Hopf monads and comonads. The example of graded
coalgebras is given its own Section~\ref{sec:exampl-dg-coalg}.

The presentation of Sections~\ref{background}, \ref{existmeascoal},
\ref{enrichmonscomons} and part of Section~\ref{actions} is similar to that
found in~\cite{2012arXiv1205.6450V,2014arXiv1411.3038V}. Soon after the first
version of this manuscript was made public, the
preprint~\cite{2015arXiv151001797P} appeared, containing some overlapping
material.

\section{Background}\label{background}
Let us start the section with a few words on terminology and notation around
monoidal categories, for which \cite{BraidedTensorCats}~is a standard reference.
Throughout the paper, the tensor product and unit object of
monoidal categories will be denoted by $\otimes$ and $I$, and the associativity and unit
constraints will be omitted in many occasions (something that is allowed by the
coherence theorem for monoidal categories). A \emph{left closed} monoidal category
$\mathcal{V}$ will be one for which the functor $(-\otimes X)$ has a right
adjoint $[X,-]$, for all objects $X$; the resulting functor $[-,-]$ is called
the \emph{left internal hom}. Symmetrically, a \emph{right closed} monoidal category is one
for which each $(X\otimes-)$ has a right adjoint $[X,-]'$, called the \emph{right
internal} hom. Braidings will be denoted by the letter $c$,
and they induce a \emph{biclosed} monoidal structure on $\ca{V}$,
should it be right or left closed.

A dual pair in a monoidal category is a pair of objects $X$,
$X^\vee$ with two morphisms $\ev\colon X^\vee\otimes X\to I$ and $\coev\colon
I\to X\otimes X^\vee$, satisfying two ``triangular equalities''; $X^\vee$ is
said to be a \emph{left dual} of $X$, and, reciprocally, $X$ a \emph{right dual} of
$X^\vee$. A dual pair induces an adjunction $(-\otimes X)\dashv(-\otimes
X^\vee)$, and $Y\otimes X^\vee$ is the left internal hom from $X$ to
$Y$. When the monoidal category is braided, right duals can be obtained from
left duals, via the braiding, and the adjectives ``left'' and ``right'' may be
dropped.  For example, a $k$-module $X$ has a dual if and only if it is
projective and finitely presentable, in which case the dual is the usual linear dual $X^*$.

An object of a braided monoidal category $\mathcal{V}$ is \emph{dualizable} if
it has a dual (left dual, or equivalently, right dual). Given a functor $U\colon
\mathcal{C}\to\mathcal{V}$, an object $X\in\mathcal{C}$ is
\emph{$U$-dualizable}, or simply dualizable when $U$ is implicit, if $U(X)$ is
dualizable in $\mathcal{V}$.

A \emph{monoidal functor} between monoidal categories $\mathcal{V}$ and
$\mathcal{W}$ will be a functor $F\colon\mathcal{V}\to \mathcal{W}$ equipped
with a transformation $F_{2,X,Y}\colon F(X)\otimes F(Y)\to F(X\otimes Y)$ and a
morphism $F_0:I\to F(I)$ satisfying coherence axioms; see, eg \cite[\S
1]{BraidedTensorCats}. Other names in use for this notion are \emph{tensor
  functor} or \emph{lax monoidal functor}. The dual notion will be called an
\emph{opmonoidal functor}, ie $F$ is equipped with a transformation
$F_{2,X,Y}\colon F(X\otimes Y)\to F(X)\otimes F(Y)$ and a morphism $F_0\colon
F(I)\to I$, satisfying coherence axioms. Other names in use for this notion are
\emph{colax monoidal functor} and \emph{oplax monoidal functor}. If $F_{2,X,Y}$
and $F_0$ are isomorphisms (resp. identities), $F$ is a \emph{strong}
(resp. \emph{strict}) monoidal functor, and it is moreover \emph{braided}
monoidal when it preserves the braiding in the appropriate sense.

Throughout the paper, we employ the well-known fact that the right adjoint of a strong
monoidal functor between monoidal categories has a unique monoidal structure
such that the unit and counit of the adjunction become monoidal natural
transformations.
This generalises to the case of
parametrised adjoints; a higher dimension version of the following proposition appeared
in~\cite[Prop.~2]{Monoidalbicats&hopfalgebroids}.
\begin{prop}\label{param}
Suppose $F:\ca{B}\times\ca{C}\to \ca{D}$
and $G:\ca{C}^\mathrm{op}\times\ca{D}\to\ca{B}$ are
parametrised adjoints, ie $F(-,C)\dashv G(C,-)$ for
all $C$, and all the categories
are monoidal. Then there is a bijection between opmonoidal structures on $F$
and monoidal structures on $G$.
\end{prop}
Let $\ca{V}$ be a braided monoidal closed category. Recall from~\cite[\S 5]{BraidedTensorCats}
that the braiding endows $\otimes\colon\ca{V}\times\ca{V}\to\ca{V}$ with a
strong monoidal structure, given by
\begin{equation}
  \label{eq:2}
  1\otimes c_{Y\otimes Z\otimes W}\colon
  X\otimes Y\otimes Z\otimes W\xrightarrow{\;\sim\;}X\otimes Z\otimes Y\otimes W
  \quad\textrm{and}\quad
  I\cong I\otimes I.
\end{equation}
By definition, the internal hom is a parametrised right adjoint to
$(-\otimes A)\cong(A\otimes -)$. As a result,
the bifunctor $[-,-]:\ca{V}^\mathrm{op}\times\ca{V}\to\ca{V}$ has a monoidal structure, by
Proposition~\ref{param}. In the case that the braiding is a symmetry, both
the tensor product and internal hom become braided monoidal functors.

\subsection{Monoids, comonoids and bimonoids}
\label{sec:mono-comon-bimon}

A \emph{monoid} in a monoidal category $\mathcal{V}$ is an object $A$ equipped with a
multiplication and unit morphisms $\mu\colon H\otimes H\to H\leftarrow I:\iota$
that satisfy the usual associativity and unit axioms, that we depict.
\begin{equation}
  \label{eq:1}
  \diagram
  A\otimes A\otimes A\ar[r]^-{\mu\otimes1}\ar[d]_{1\otimes \mu}&
  A\otimes A\ar[d]^\mu\\
  A\otimes A\ar[r]^-\mu&
  A
  \enddiagram
  \qquad
  \diagram
  A\ar[r]^-{\iota\otimes 1}\ar[dr]_1&
  A\otimes A\ar[d]^\mu &
  A\ar[l]_-{1\otimes \iota}\ar[dl]^1\\
  &A&
  \enddiagram
\end{equation}
A morphism of monoids $(A,\iota,\mu)\to(A',\iota',\mu')$ is a morphism $f\colon
A\to A'$ in $\mathcal{V}$ compatible with multiplication and unit; ie such that
$\mu'\cdot(f\otimes f)=f\cdot\mu$ and $f\cdot \iota=\iota'$.

A \emph{comonoid} in $\mathcal{V}$ is a monoid in the opposite monoidal category
$\mathcal{V}^{\mathrm{op}}$; it consists of an object $C$ with a
comultiplication $\Delta\colon C\to C\otimes C$ and a counit $\varepsilon\colon
C\to I$ satisfying axioms dual to those of a monoid. Morphisms of comonoids are
defined in a dual fashion to morphisms of monoids.

The categories of monoids and comonoids in a monoidal category $\ca{V}$
will be denoted, respectively, by $\Mon(\mathcal{V})$ and
$\Comon(\mathcal{V})$.
For $\ca{V}$ braided, these categories are monoidal: if $A$ and $A'$ are monoids, then
$A\otimes A'$ has a monoid structure with multiplication $(\mu\otimes
\mu')\cdot(A \otimes c_{A',A}\otimes A')$.
The respective forgetful functors into $\mathcal{V}$ are strict monoidal.
These categories need not support a braiding unless $\mathcal{V}$ is symmetric,
as explained below.

\begin{rmk}
  \label{rmk:2}
 The monoidal category $\Comon(\ca{V})$ on a braided monoidal
 category $\mathcal{V}$ has a braiding given by $c_{AB}\colon A\otimes B\to
 B\otimes A$, ie $c_{AB}$ is a morphism of comonoids, if $c$ is a
 symmetry. The analogous result holds for $\Mon(\ca{V})$.
\end{rmk}

Monoidal functors preserve monoids, in the sense that, given a monoid $A$
and a monoidal functor $F\colon\mathcal{V}\to\mathcal{W}$, then $F(A)$ is a
monoid with multiplication $F(\mu)\cdot F_{2,A,A}$ and unit $F(\iota)\cdot F_0$.
We denote the induced functor between the categories of monoids by $\Mon(F):\Mon(\ca{V})\to
\Mon(\ca{W})$. Dually, opmonoidal functors preserve comonoids.

As an example, the monoidal structure of $\Mon(\ca{V})$ and $\Comon(\ca{V})$
can be deduced from $\otimes$ being strong monoidal.
Also, the internal hom functor induces
\begin{equation}\label{MonHom}
  \Mon[-,-]\colon
  \Comon(\ca{V})^\mathrm{op} \times\Mon(\ca{V})
  \longrightarrow
  \Mon(\ca{V})\qquad
  (C,A)\mapsto [C,A].
\end{equation}
In particular, whenever $C$ is a comonoid and $A$ a monoid, the object $[C,A]$
is endowed with the structure of a monoid, sometimes called the
\emph{convolution} monoid structure. We record for later reference:
\begin{lem}\label{Hsymmetriclax}
The internal hom functor $[-,-]$
of a braided monoidal closed category $\mathcal{V}$
induces a
functor $\Comon(\ca{V})^\mathrm{op}
\times\Mon(\ca{V})\to\Mon(\ca{V})$. When the braiding is a symmetry, the domain
and codomain are symmetric monoidal categories and this functor is braided.
\end{lem}

In any braided monoidal category $\ca{V}$, the braiding allows us to define
\emph{opposite} monoids and comonoids. In contrast to
the case of a symmetric monoidal category (ie when $c_{X,Y}^{-1}=c_{Y,X}$),
there are two choices of opposite, one that employs $c$ and the other that
employs $c^{-1}$. If $(A,\iota,\mu)$ is a monoid, we denote by
$A^{\mathrm{op},c}$ and $A^{\mathrm{op},c^{-1}}$ the monoids with
multiplication $\mu\cdot c_{A,A}$ and $\mu\cdot c_{A,A}^{-1}$ respectively. Similarly, if
$(C,\varepsilon,\Delta)$ is a comonoid, we denote by $C^{\mathrm{cop},c}$ and
$C^{\mathrm{cop},c^{-1}}$ the comonoids with comultiplication
$c_{C,C}\cdot\Delta$ and $c_{C,C}^{-1}\cdot\Delta$. If $c$ is a symmetry,
we clearly only have $A^\op$ and $C^\mathrm{cop}$.
A monoid $A$ is then called \emph{commutative} if $A=A^\op$,
and dually for a \emph{cocommutative} comonoid $C=C^\mathrm{cop}$.

\begin{df}
\label{df:4}
\begin{enumerate}
\item\label{item:39}
A \emph{bimonoid} in a braided $\mathcal{V}$ is an object $B$ with a
monoid structure $(\iota,\mu)$ and a comonoid structure $(\varepsilon,\Delta)$
such that $\varepsilon\colon B\to I$ and $\Delta\colon B\to B\otimes B$
are monoid morphisms, where $B\otimes B$ is a monoid with the structure
described earlier.
\item\label{item:40}
An \emph{antipode} for a bimonoid
$(H,\iota,\mu,\varepsilon,\Delta)$ is a morphism $s\colon H\to H$ 
for which $\mu\cdot(s\otimes1_H)\cdot\Delta=\iota\cdot\varepsilon=
\mu\cdot(1_H\otimes s)\cdot\Delta$.
In other words, is an inverse for $1_H$ in
the convolution monoid induced by the bimonoid $H$. A bimonoid is called a
\emph{Hopf monoid} if it has an antipode.
\item\label{item:41}
An \emph{opantipode} is an inverse for $1_H$ in the convolution monoid induced by
the comonoid $H^{\mathrm{cop},c^{-1}}$ and the monoid $H$; see~\cite[\S 5.2]{MR2793022}.
The bimonoid $H$ is called an \emph{op-Hopf monoid} if it has an opantipode.
\item\label{item:42}
Finally, for any monoid $A$ in a monoidal category $\mathcal{V}$, we shall denote its
category of left modules by $\Mod_{\mathcal{V}}(A)$; its objects are pairs
$(M,\nu)$ where $M$ is an object of $\mathcal{V}$ and $\nu\colon A\otimes M\to
M$ is an action of $A$, by which we mean that it satisfies the usual module
axioms. Dually, we denote the category of left comodules over a comonoid $C$ by
$\Comod_{\mathcal{V}}(C)$.
\end{enumerate}
\end{df}

\subsection{Accessible and locally presentable monoidal categories}\label{admcats}

This section compiles some facts about accessible and locally presentable monoidal categories,
present in~\cite{MonComonBimon}, that will be useful later. We refer the reader
to~\cite{MR1031717} or \cite{LocallyPresentable} for the theory of accessible
and locally presentable categories. We limit ourselves here to mentioning only
a few facts. An object $X$ of a category $\mathcal{C}$ is
\emph{$\kappa$-presentable} if the representable $\mathcal{C}(X,-)$ preserves
$\kappa$-filtered colimits; here $\kappa$ is a regular cardinal. The category
$\mathcal{C}$ is \emph{$\kappa$-accessible} if it has $\kappa$-filtered colimits, there is, up to
isomorphism, a small set of $\kappa$-presentable objects, and each object of $\mathcal{C}$ is the
colimit of a $\kappa$-filtered diagram of $\kappa$-presentable objects. In
particular, $\kappa$-presentable objects form a small dense subcategory of
$\mathcal{C}$. If $\mathcal{C}$ is cocomplete, one says that it is \emph{locally
$\kappa$-presentable}. A category is \emph{accessible} or \emph{locally presentable} if it is
$\kappa$-accessible or locally $\kappa$-presentable for some regular cardinal
$\kappa$. Locally presentable categories are automatically complete.
Also, if $\ca{C}$ is accessible or locally presentable, than so is the functor
category $[\ca{A},\ca{C}]$ for any small $\ca{A}$.

A functor between $\kappa$-accessible categories is \emph{$\kappa$-accessible}
if it preserves $\kappa$-filtered colimits. A functor between two accessible
categories is said to be \emph{accessible} if it is $\kappa$-accessible for some
$\kappa$. This makes sense because given two (or a small set $L$ of) accessible
categories, there exists a regular cardinal $\kappa$ such that both (all the
elements of $L$) are
$\kappa$-accessible (see~\cite[2.4.9]{MR1031717}). Clearly, if $F$ is
$\kappa$-accessible and $\kappa\leq\mu$, then $F$ is $\mu$-accessible.

\begin{rmk}
\label{rmk:7}
The following two facts about functors between locally presentable categories
will be needed in later sections.
\begin{enumerate}
\item\label{item:7}
Any cocontinuous functor between locally presentable categories is a left
adjoint. In fact, by~\cite[5.33]{Kelly}, any cocontinuous functor between
cocomplete categories whose domain has a small dense category is a left
adjoint.
\item\label{item:8}
Any continuous accessible functor between locally presentable categories
is a right adjoint, by~\cite[Satz~14.6]{GabrielUlmer}.
\end{enumerate}
\end{rmk}

\begin{df}\label{def:accessible monoidal cat}
A \emph{$\kappa$-accessible monoidal category} is a monoidal
category $\mathcal{V}$ that satisfies the following:
it is a $\kappa$-accessible category; the
tensor product is a $\kappa$-accessible functor; the unit object is
$\kappa$-presentable; and, $\kappa$-presentable objects are closed under
the tensor product.
A \emph{locally presentable monoidal category} is an accessible monoidal
category that is cocomplete.
\end{df}

Locally presentable monoidal categories that are
symmetric are called \emph{admissible} monoidal categories
in~\cite{MonComonBimon}.
Examples include the monoidal category of modules over
a commutative ring; the monoidal category of chain complexes 
over a commutative
ring; any locally presentable cartesian closed category, such as for example
Grothendieck toposes.
In fact, the results obtained in the current paper require more relaxed conditions
on the base monoidal category $\ca{V}$, ie local presentability
of its underlying category, as we explain below.

Any monoidal biclosed structure $(\otimes,I,[-,-])$ on an accessible category
$\mathcal{V}$ is automatically an accessible monoidal category. Indeed, $I$ is
$\kappa$-presentable for some $\kappa$, by~\cite[2.3.12]{MR1031717}. If
$X$, $Y$ are $\kappa$-presentable objects, then so is $X\otimes Y$, since
$\mathcal{V}(X\otimes Y,-)\cong\mathcal{V}(X,[Y,-])$ is the composition of two
$\kappa$-accessible functors: $\mathcal{V}(X,-)$ by hypothesis, and $[Y,-]$
as a right adjoint~\cite[2.4.8]{MR1031717}. Clearly, the tensor product preserves
all colimits, as it has a right adjoint.

Later we will use the following easy lemma. Recall that a left adjoint $U\dashv
R$ is of \emph{codescent type} if the components $\eta_X\colon X\to RU(X)$ of
the unit are equalisers. If this is the case, $\eta_X$ is necessarily the
equaliser of the pair of morphisms $\eta_{RU(X)},RU\eta_X\colon
RU(X)\rightrightarrows RURU(X)$.
\begin{lem}
  \label{l:9}
  Let $U\dashv R\colon\mathcal{B}\to\mathcal{A}$ be an adjunction of
  codescent type between accessible categories. Then $X\in\mathcal{A}$ is
  presentable if and only if $U(X)\in\mathcal{B}$ is presentable. Furthermore,
  the presentability degree of $X$ is at least the maximum of presentability
  degree of $U(X)$ and the accessibility degree of $RU$.
\end{lem}
\begin{proof}
  The functor $\mathcal{B}(U(X),-)\cong\mathcal{A}(X,R-)$ is accessible if $X$
  is presentable, being the composition of the accessible functors
  $\mathcal{A}(X,-)$ and $R$, so the direct implication is obvious. For the converse,
  suppose that $U(X)$ is presentable, so $\mathcal{A}(X,R-)$ is
  accessible. There is an equaliser, natural in $Y$, by the hypothesis that the
  adjunction is of codescent type:
  \begin{equation}
    \label{eq:75}
    \mathcal{A}(X,Y)\rightarrowtail{}\mathcal{A}(X,RU(Y))\rightrightarrows
    \mathcal{A}(X,RURU(Y))
  \end{equation}
  Thus, $\mathcal{A}(X,-)$ is an equaliser of accessible functors
  $\mathcal{A}\to\mathbf{Set}$, and, therefore, it is $\kappa$-accessible by
  \cite[Prop.~2.4.5]{MR1031717}, where $\kappa$ can be taken as the maximum of
  the accessibility degree of $\mathcal{A}(X,-)$ and $RU$.
\end{proof}
\begin{cor}
  \label{cor:8}
  If $\mathsf{G}$ is an accessible comonad on an accessible category
  $\mathcal{C}$, then:
  \begin{enumerate*}
  \item \label{item:19} The category $\mathcal{C}^G$ of Eilenberg--Moore coalgebras
    is accessible.
  \item \label{item:20} The forgetful functor
    $U\colon\mathcal{C}^G\to\mathcal{C}$ is accessible.
  \item \label{item:21} The presentable $G$-coalgebras are those whose
    underlying object is presentable in $\mathcal{C}$. Furthermore, the
    presentability degree of a $G$-coalgebra $M$ is at least the maximum of the
    presentability degree of $U(M)$ and the accessibility degree of $G$.
  \end{enumerate*}
\end{cor}
\begin{proof}
  The first two claims hold by~\cite[5.1.6]{MR1031717}, while the last is an
  instance of Lemma~\ref{l:9}, as comonadic functors are of codescent type.
\end{proof}

Compare the following result with \cite[\S 2]{MonComonBimon}.
\begin{prop}
  \label{moncomonadm}
  Suppose $\ca{V}$ is an accessible (resp. locally presentable) monoidal category. Then both
  $\Mon(\ca{V})$ and $\Comon(\ca{V})$ are accessible (resp. locally presentable)
  categories, and
  the respective forgetful functors are accessible.
\end{prop}
\begin{proof}
  Both the category of monoids and comonoids can be constructed from
  $\mathcal{V}$ by using products, inserters and equifiers;
  see~\cite{2categoricallimits} for a description of these limits. Then,
  \cite[5.1.6]{MR1031717} implies that the categories of monoids and comonoids,
  and the respective forgetful functors, are accessible. Now suppose that
  $\mathcal{V}$ is locally presentable. In that case,
  $\Comon(\mathcal{V})$ is cocomplete and $\Mon(\mathcal{V})$ is complete,
  and therefore both are locally presentable (see~\cite[6.1.4]{MR1031717}).
\end{proof}

As an application, consider the category
$\Comon(\ca{V})$ for a locally presentable braided monoidal closed category
$\ca{V}$. Then the functor $(-\otimes C)$ with domain a locally presentable
category is cocontinuous, by the commutative
diagram below, thus it has a right adjoint by Remark \ref{rmk:7}.
The same argument holds for $(C\otimes-)$, so $\Comon(\mathcal{V})$ is a locally
presentable monoidal biclosed category (see also~\cite[3.2]{MonComonBimon}).
\[
\xymatrix@R=.5cm{\mathbf{Comon}(\ca{V})\ar[r]^-{-\otimes
C}\ar[d]_-U & \mathbf{Comon}(\ca{V})\ar[d]^-U\\
\ca{V}\ar[r]^-{-\otimes UC} & \ca{V}}
\]

\begin{ex}
  \label{ex:14}
  The category $\mathbf{gVect}_{\mathbb{Z}}$ of $\mathbb{Z}$-\emph{graded} $k$-vector spaces, being
  equivalent to the category of functors from the discrete category $\mathbb{Z}$
  into $\mathbf{Vect}$, is locally finitely presentable. Furthermore, it is
  locally finitely presentable as a monoidal category, with the tensor product
  $(V\otimes W)_n=\sum_{n=i+j}V_i\otimes W_j$ and unit $I$ equal to $k$
  concentrated in degree $0$. There is a symmetry on $\mathbf{gVect}_{\mathbb{Z}}$,
  given on homogeneous elements $x\in V_a$, $y\in W_b$ by $s_{V,W}(x\otimes y)
  = (-1)^{ab}y\otimes x$. The internal
  hom is $[V,W]_n=\prod_i\mathrm{Hom}_k(V_i,W_{i+n})$. The category
  $\mathbf{gVect}_{\mathbb{N}}$ of $\mathbb{N}$-graded $k$-vector spaces is a
  locally presentable monoidal subcategory of $\mathbf{gVect}_{\mathbb{Z}}$.
\end{ex}
\begin{ex}
  \label{ex:13}
  Let $\mathbf{dgVect}_{\mathbb{Z}}$ be the category of \emph{chain
  complexes} of vector spaces, or \emph{differential graded} vector spaces.
  This is a locally finitely
  presentable category; the finitely presentable objects are the bounded chain complexes of
  finite-dimensional vector spaces.
  There exists an obvious forgetful functor
  $\mathbf{dgVect}_{\mathbb{Z}}\to\mathbf{gVect}_{\mathbb{Z}}$ that forgets the differential,
  which preserves limits and colimits and is conservative. It is a
  classical fact that this is a symmetric monoidal closed category and the said
  forgetful functor is strict monoidal; in other words, if $V$ and $W$ are dg vector
  spaces, then their graded tensor product
  can be equipped with differentials, which are compatible
  with the relevant natural transformations: the associativity and unit
  constraints and the symmetry. Explicit formulas for these differentials can be found in any
  homological algebra textbook.

  The full monoidal
  subcategory $\mathbf{dgVect}_{\mathbb{N}}$ of non-negatively graded chain
  complexes is locally presentable too.
\end{ex}

\subsection{Enriched categories}
\label{sec:enriched-categories}
It might be helpful to recall the definitions of enriched categories,
functors, and so on, that will be used in the article. We only give an outline;
detailed definitions can be found in~\cite{Kelly}. The base of enrichent will be
a monoidal category $(\mathcal{V},I,\otimes)$, that in many instances will be
assumed to be braided, closed, cocomplete or even finitely presentable. A
$\mathcal{V}$-category $\mathcal{C}$ consists of objects $X,Y$, etc., and
objects $\mathcal{C}(X,Y)$ of $\mathcal{V}$, for each pair of object $X,Y$. It
is equipped with composition morphisms
$\mathcal{C}(Y,Z)\otimes\mathcal{C}(X,Y)\to\mathcal{C}(X,Z)$ and identity
morphisms $I\to \mathcal{C}(X,X)$ that satisfy associativity and identity
axioms. A $\mathcal{V}$-functor $F\colon \mathcal{C}\to\mathcal{D}$ sends
objects of $\mathcal{C}$ to objects of $\mathcal{D}$, and is given on enriched
homs by morphisms $F_{X,Y}\colon \mathcal{C}(X,Y)\to\mathcal{D}(FX,FY)$ in
$\mathcal{V}$, that are compatible with composition and identities. A
$\mathcal{V}$-natural transformation $\tau$ from $F$ to another
$\mathcal{V}$-functor $G$ consists of a family of morphisms $\tau_{X}\colon
I\to\mathcal{D}(FX,GX)$ that satisfy naturality axioms.

When the monoidal category $\mathcal{V}$ has a braiding $c_{X,Y}\colon X\otimes
Y\to Y\otimes X$, one can consider the tensor product of two
$\mathcal{V}$-categories $\mathcal{C}$ and $\mathcal{D}$. This is a
$\mathcal{V}$-category $\mathcal{C}\otimes\mathcal{D}$ with objects
$\mathrm{ob}\mathcal{C}\times\mathrm{ob}\mathcal{D}$, and enriched homs
\begin{equation}
  \label{eq:11}
  (\mathcal{C}\otimes \mathcal{D})
  \bigl(
  (C,D),(C',D')
  \bigr)
  =\mathcal{C}(C,C')\otimes \mathcal{D}(D,D').
\end{equation}
The braiding is used to define the composition of $\mathcal{C}\otimes
\mathcal{D}$ by
\begin{multline}
  \label{eq:12}
  \mathcal{C}(C',C'')\otimes\mathcal{D}(D',D'')\otimes\mathcal{C}(C,C') \otimes
  \mathcal{D}(D,D')
  \xrightarrow{1\otimes c\otimes 1}\\
  \longrightarrow 
  \mathcal{C}(C',C'')\otimes \mathcal{C}(C,C') \otimes\mathcal{D}(D',D'')\otimes
  \mathcal{D}(D,D')
  \longrightarrow \mathcal{C}(C,C'')\otimes\mathcal{D}(D,D'').
\end{multline}

\subsection{Kleisli categories}\label{Kleislicats}
In this section, we will describe some known facts regarding
Kleisli categories for monoidal and enriched monads. We gather these facts here,
in order to refer to them later.

In general, if $(\mathsf{T},\eta,\mu)$ is a monad on an ordinary category $\mathcal{V}$,
its Kleisli category -- denoted by $\mathcal{V}_{\mathsf{T}}$ or
$\operatorname{Kl}(\mathsf{T})$ -- has the same objects as
$\mathcal{V}$ and homs $\operatorname{Kl}(\mathsf{T})(X,Y)=\mathcal{V}(X,TY)$;
the composition uses the multiplication $\mu$ of $\mathsf{T}$ and the identity
morphism of an object $X$ is the unit $\eta_X\colon X\to \mathsf{T}X$;
for more details see~\cite[\S VI.5]{MacLane}. There is a bijective on objects functor
$F_{\mathsf{T}}\colon\mathcal{V}\to\operatorname{Kl}(\mathsf{T})$ that sends a
morphism $f\colon X\to Y$ to $\eta_{Y}\cdot f$.

If $\mathcal{V}$ is a monoidal category and $(\mathsf{T},\eta,\mu)$ has a monoidal
monad structure, ie is a (lax) monoidal endofunctor on $\ca{V}$ with
$\eta$, $\mu$ monoidal, then $\operatorname{Kl}(\mathsf{T})$
carries a monoidal structure that makes $F_{\mathsf{T}}$ a strict monoidal functor;
in other words, we can tensor objects of the Kleisli category as we do in $\mathcal{V}$.
If $\mathcal{V}$ has a braiding $c$ and the functor $T$ is braided monoidal, then there exists a
braiding on $\operatorname{Kl}(\mathsf{T})$ that makes $F_{\mathsf{T}}$ a
braided monoidal functor.

We will now consider $\mathcal{V}$-enriched monads on a braided monoidal closed
category $\mathcal{V}$. In other words, $\mathcal{V}$ is regarded as a
$\mathcal{V}$-category, with enriched hom-objects $[A,B]$. A $\mathcal{V}$-monad
$\mathsf{T}=({T},\eta,\mu)$ on $\mathcal{V}$ consists of a endo-$\mathcal{V}$-functor $T$
and unit $\eta$ and multiplication $\mu$ that are $\mathcal{V}$-natural
transformations, and that form a monad on the ordinary category
$\mathcal{V}$. The Kleisli
$\mathcal{V}$-category $\mathcal{V}_\mathsf{T}$ of $\mathsf{T}$ has the same
objects as $\mathcal{V}$,
and enriched homs $\mathcal{V}_\mathsf{T}(X,Y)=[X,\mathsf{T}Y]$. Composition and
identities are given by
\begin{gather}
  \label{eq:54}
  [Y,TZ]\otimes[X,TY]\xrightarrow{T\otimes1}[TY,T^2Z]\otimes[X,TY]
  \xrightarrow{\mathrm{comp}}
  [X,T^2Z]\xrightarrow{[X,\mu_Z]}[X,TZ]\\
  I\xrightarrow{\mathrm{id}}[X,X]\xrightarrow{[X,\eta_X]}[X,TX].
\end{gather}
On the other hand, the $\mathcal{V}$-category of $\mathsf{T}$-algebras,
denoted by $\mathcal{V}^\mathsf{T}$, has objects the usual
$\mathsf{T}$-algebras, and  enriched homs $\mathcal{V}^T((A,a),(B,b))$
the equaliser of the morphisms
\begin{equation}
  \label{eq:55}
  [A,B]\xrightarrow{T}[TA,TB]\xrightarrow{[TA,b]}[TA,B]\quad\text{and}\quad
  [A,B]\xrightarrow{[a,B]}[TA,B]
\end{equation}
with composition induced by that of the $\mathcal{V}$-category
$\mathcal{V}$. There is a full and faithful ``comparison'' $\mathcal{V}$-functor
\begin{equation}
  K\colon \mathcal{V}_\mathsf{T}\longrightarrow\mathcal{V}^\mathsf{T}\label{eq:76}
\end{equation}
given on objects by $X\mapsto\mathsf{T}X$ and on
homs by the isos $[X,TY]\cong\mathcal{V}^T(TX,TY)$ induced by the the
morphisms
\begin{equation}
  \label{eq:57}
  [X,TY]\xrightarrow{T}[TX,T^2Y]\xrightarrow{[TX,\mu_Y]}[TX,TY].
\end{equation}
As is the case for any $\mathcal{V}$-functor, $K$ gives monoid morphisms
between endo-homs
\begin{equation}
  \mathcal{V}_T(X,X)=[X,TX]\cong\mathcal{V}^T(TX,TX)\rightarrowtail{}[TX,TX];
  \label{eq:58}
\end{equation}
the multiplication is composition, which in the case of $\mathcal{V}_T(X,X)$
was described in~\eqref{eq:54}.

\section{Actions of monoidal categories and enrichment}\label{actions}
Recall that a \emph{left action} of a monoidal category
$\ca{V}={(\ca{V},\otimes,I,a,l,r)}$ on a category $\ca{D}$ is given
by a functor $*:\ca{V}\times\ca{D}\to\ca{D}$,
a natural isomorphism with components
$\alpha_{XYD}\colon X*(Y*D)\xrightarrow{\sim}(X\otimes Y)*D$
and a natural isomorphism with components
${\lambda}_D :I*D\xrightarrow{\sim} D$,
satisfying the commutativity of diagrams similar to those of a monoidal category;
see~\cite[\S 1]{AnoteonActions} for a full description. Another way of describing a
left action of $\mathcal{V}$ is by a strong monoidal functor
$\mathcal{V}\to\mathrm{End}(\mathcal{D})$ into the strict monoidal category of
endofunctors of $\mathcal{D}$, whose tensor product is given by composition.

The most important fact here for us, explained in detail in~\cite{AnoteonActions},
is that to give a category $\ca{D}$ and a left action of a monoidal
category $\ca{V}$ with a right adjoint for each $(-*D)$ is to give
a $\ca{V}$-category $\ca{D}$.

\begin{prop}
  \label{actionenrich}
  Suppose $*\colon \mathcal{V}\times \mathcal{D}\to\mathcal{D}$ is a left
  action of the monoidal category $\mathcal{V}$ with the property that $(-*D)$
  has a right adjoint $H(D,-)$, for each $D\in\mathcal{D}$;
  \begin{equation}
    \label{eq:18}
    \mathcal{D}(X*D,E)\cong\mathcal{V}(X,H(D,E)).
  \end{equation}
  Then, there exists a
  $\mathcal{V}$-enriched category $\underline{\mathcal{D}}$ with underlying
  category $\mathcal{D}$ and hom-objects $\underline{\mathcal{D}}(D,E)=H(D,E)$.
When $\mathcal{V}$ is left closed, this establishes an equivalence between left
actions of $\mathcal{V}$ and tensored $\mathcal{V}$-categories.
\end{prop}
The proof of the existence of the composition
$H(B,C)\otimes H(A,B)\to H(A,C)$ and the identity morphisms
$I\to H(A,A)$ satisfying the usual axioms of
enriched categories is easily deduced from the correspondence of
arrows under the adjunction~(\ref{eq:18}) and the
action axioms.

Moreover, when the monoidal
category $\ca{V}$ is symmetric, then the opposite of a $\mathcal{V}$-category
can be defined in the usual way, so we have that
$\ca{D}^\mathrm{op}$
is also enriched in $\ca{V}$, with the same objects and
hom-objects
$\ca{D}^\mathrm{op}(B,A)=\ca{D}(A,B)$. Notice that if $\mathcal{V}$ is braided but not
symmetric, there are two different choices of opposite $\mathcal{V}$-category,
one using the braiding and the other using its inverse.


\begin{rmk}
  The statement of Proposition~\ref{actionenrich} mentions tensored
  $\mathcal{V}$-categories over a left closed monoidal category $\mathcal{V}$.
  These are $\mathcal{V}$-categories $\mathcal{C}$ for with
  $\mathcal{V}$-natural isomorphisms
  $\mathcal{C}(X*C,D)\cong[X,\mathcal{C}(C,D)]$ where $[-,-]$ is the
  left internal hom of $\mathcal{V}$, $X\in\mathcal{V}$ and $C,D\in\mathcal{C}$.
  \label{rmk:4}
\end{rmk}

\begin{ex}
  In addition to the tensored $\mathcal{V}$-categories mentioned above,
  examples of actions of a monoidal left closed category $\mathcal{V}$ are
  provided by the cotensored $\mathcal{V}$-categories, or rather, by a
  $\mathcal{V}$-categories $\mathcal{A}$ with chosen cotensor products.
  Dually to tensors, a cotensor
  product of $A\in \mathcal{A}$ by $X\in\mathcal{V}$ is an object
  $\{X,A\}\in\mathcal{A}$ with a morphism
  $X\to\mathcal{A}(\{X,A\},A)$
  that induces isomorphisms
  $\mathcal{A}(B,\{X,A\})\cong[X,\mathcal{A}(B,A)]$. Then, the canonical
  isomorphisms $\{X,\{Y,A\}\}\cong\{X\otimes Y,A\}$ and $\{I,A\}\cong A$ make the
  functor $\{-,-\}^{\mathrm{op}}$ into a left action of $\mathcal{V}$ on
  $\mathcal{A}_\circ^{\mathrm{op}}$, the opposite of the underlying category of
  $\mathcal{A}$. For example, the left internal hom $[-,-]^\op$ is a left action of $\mathcal{V}$ on
  $\mathcal{V}^{\mathrm{op}}$.\label{ex:1}
\end{ex}

In Theorem~\ref{thisthm} we shall give a monoidal version of
Proposition~\ref{actionenrich}, but before that we need the following easy
theorem. First recall that given a braided monoidal category
$\mathcal{V}$, a $\mathcal{V}$-\emph{enriched monoidal structure} on a
$\mathcal{V}$-category $\mathcal{A}$ consists of a $\mathcal{V}$-functor
$\otimes\colon \mathcal{A}\otimes\mathcal{A}\to\mathcal{A}$, an object $I\in
\mathcal{A}$ and $\mathcal{V}$-natural isomorphisms $(X\otimes Y)\otimes Z\cong
X\otimes(Y\otimes Z)$, $I\otimes X\cong X\cong X\otimes I$, such that the
underlying functor $\otimes_{\circ}$ together with $I$ and these isomorphisms
form a monoidal structure on the ordinary category $\mathcal{A}_{\circ}$. One
says that $\mathcal{A}$ is a \emph{monoidal $\mathcal{V}$-category}.
It is not hard to see how this definition establishes the following equivalence.

\begin{thm}\label{mainthm}
Let $\ca{V}$ be a braided monoidal category. Suppose $\ca{A}$ is a
$\ca{V}$-category equipped with a $\mathcal{V}$-functor
$T\colon\mathcal{A}\otimes\mathcal{A}\to\mathcal{A}$ and an object $J$.
There is a bijection between:
\begin{enumerate}
\item \label{item:4} Monoidal $\ca{V}$-category structures on $\ca{A}$ with
  tensor product $T$ and unit $J$;
\item \label{item:5} Extensions of $(\mathcal{A}_\circ,T_\circ,J)$ to a monoidal
  category such that the morphisms $T_{ABCD}\colon
  \mathcal{A}(A,C)\otimes\mathcal{A}(B,D)\to\mathcal{A}(T(A,B),T(C,D))$ and
  $\mathrm{id}\colon I\to \mathcal{A}(J,J)$ make the enriched-hom functor
  ${\ca{A}}(-,-):\ca{A}_\circ^\op\times\ca{A}_\circ\to\ca{V}$ into a monoidal
  functor.
\end{enumerate}
Furthermore, if $\mathcal{A}_0$ is braided, with braiding $c$, the monoidal
$\mathcal{V}$-category of~(\ref{item:4}) is braided with braiding $c$ if and
only if ${\mathcal{A}}(-,-)$ is a braided monoidal functor from
$(\mathcal{A}_\circ^{\mathrm{op}}\times\mathcal{A}_\circ,c^{-1}\times c)$ to
$(\mathcal{A},c)$.
\end{thm}

\begin{df}
  \label{df:1}
  Let $*\colon\mathcal{V}\times\mathcal{A}\to\mathcal{A}$ be a left action of
  the braided monoidal category $\mathcal{V}$. If $(\mathcal{A},\diamond,J)$ is
  a monoidal category, by an \emph{opmonoidal structure} on the action we shall
  mean an opmonoidal structure on the functor $*$, where its domain has the
  product monoidal structure, that makes the natural isomorphisms $\alpha$ and
  $\lambda$ opmonoidal natural transformations. We speak of an opmonoidal action of
  $\mathcal{V}$.
\end{df}
In more explicit terms, an opmonoidal structure on the action $*$ consists of a
morphism and a natural transformation
\begin{equation}
  \label{eq:19}
  \xi_0\colon I*J\longrightarrow J
  \qquad
  \xi_{XYAB}\colon(X\otimes Y)*(A\diamond B)\longrightarrow (X*A)\diamond(Y*B)
\end{equation}
that make the diagrams in Figure~\ref{fig:1} commute (the associativity
constraints of both $\otimes$ and $\diamond$ are omitted); the first three
diagrams exhibit $(*,\xi,\xi_0)$ as an opmonoidal functor, while the last four
diagrams exhibit $\alpha$ and $\lambda$ as opmonoidal transformations.
\begin{figure}
  \begin{equation}
    \xymatrix@C=1.4cm{
      (X\otimes Y \otimes Z)*(A\diamond B\diamond C)
      \ar[d]|{\xi_{X\otimes Y,Z,A\diamond B,C}}
      \ar[r]^-{\xi_{X,Y\otimes Z,A,B\diamond C}}&
      (X*A)\diamond((Y\otimes Z)*(B\diamond C))
      \ar[d]|{1\diamond\xi_{YZBC}}
      \\
      ((X\otimes Y)*(A\diamond B))\diamond(Z*C)
      \ar[r]^-{\xi_{XYAB}\diamond 1}
      &
      (X*A)\diamond (Y*B)\diamond (Z*C)
    }
  \end{equation}
  \begin{equation}
    \xymatrix{
      (I\otimes X)*(J\diamond A)\ar[d]_{\xi_{IXJA}}\ar[r]^-\cong
      &
      X*A
      \\
      (I*J)\diamond(X*A)\ar[r]^-{\xi_0\diamond 1}
      &
      J\diamond(X*A)\ar[u]_\cong
    }
    \quad
    \xymatrix{
      (X\otimes I)*(A\diamond J)\ar[r]^-\cong\ar[d]_{\xi_{XIAJ}}
      &
      X*A
      \\
      (X*A)\diamond(I*J)\ar[r]^-{1\diamond\xi_0}
      &
      (X*A)\diamond J\ar[u]_\cong
    }
  \end{equation}
  \begin{equation}
    \xymatrix{
      (X\otimes X')*((Y\otimes Y')*(A\diamond A'))\ar[r]^-\alpha
      \ar[d]|{1*\xi_{YY'AA'}}&
      (X\otimes X'\otimes Y\otimes Y')*(A\diamond A')
      \ar[d]|{(1\otimes c_{X'Y}\otimes 1)*1}
      \\
      (X\otimes X')*((Y*A)\diamond(Y'*A'))
      \ar[d]|{\xi_{X,X',Y*A,Y'*A'}}
      &
      (X\otimes Y\otimes X'\otimes Y')*(A\diamond A')
      \ar[d]|{\xi_{X\otimes Y,X'\otimes Y',A,A'}}
      \\
      (X*(Y*A))\diamond(X'*(Y'*A'))
      \ar[r]^-{\alpha\diamond\alpha}
      &
      ((X\otimes Y)*A)\diamond ((X'\otimes Y')*A')
    }
  \end{equation}
  \begin{equation}
    \xymatrix@R=.5cm{
      (I\otimes I)*J\ar[r]^-{\xi_{IIJ}}\ar[d]_\cong
      &
      I*(I*J)\ar[d]^{1*\xi_0}
      \\
      I*J\ar[d]_{\xi_0}
      &
      I*J\ar[d]^{\xi_0}
      \\
      J\ar@{=}[r]
      &
      J
    }
  \end{equation}
  \begin{equation}
    \xymatrix{
      (I\otimes I)*(A\diamond B)\ar[r]^-{\xi_{IIAB}}\ar[d]_\cong
      &
      (I*A)\diamond(I*B)\ar[d]^{\lambda_A\diamond\lambda_B}
      \\
      I*(A\diamond B)\ar[r]^-\cong
      &
      A\diamond B
    }
    \qquad
    \xymatrix{
      I*J\ar[r]^{\lambda_J}\ar[d]_{\xi_0}
      &
      J\ar@{=}[d]
      \\
      J\ar@{=}[r]
      &
      J
    }
  \end{equation}
  \caption{Axioms of an opmonoidal action.}
\label{fig:1}
\end{figure}

\begin{thm}\label{thisthm}
  Suppose given a left action $*\colon \ca{V}\times \ca{A}\to\ca{A}$ of a
  braided monoidal category $\mathcal{V}$ such that $(-*A)$ has a right adjoint for
  all $A$, and let $\underline{\mathcal{A}}$ be the associated
  $\mathcal{V}$-category. Then, each opmonoidal structure on the action induces
  a monoidal $\mathcal{V}$-category structure on $\underline{\mathcal{A}}$
  with underlying monoidal category $\mathcal{A}$.
\end{thm}
\begin{proof}
  We will first construct a functor
  $T\colon\underline{\mathcal{A}}\otimes\underline{\mathcal{A}}\to\underline{\mathcal{A}}$. On
  objects it will be given by $T(A,B)=A\diamond B$; on homs it is given by
  the morphisms $T_{ABCD}\colon
  \underline{\mathcal{A}}(A,C)\otimes\underline{\mathcal{A}}(B,D)\to
  \underline{\mathcal{A}}(A\diamond B,C\diamond D)$ that are transpose to the composition
  \begin{equation}
    \label{eq:36}
    (
    \underline{\mathcal{A}}(A,C)\otimes\underline{\mathcal{A}}(B,D))*(A\diamond
    B)
    \xrightarrow{\xi}
    (\underline{\mathcal{A}}(A,C)* A)\diamond (\underline{\mathcal{A}}(B,D)*B)
    \xrightarrow{\varepsilon\diamond\varepsilon}
    C\diamond D.
  \end{equation}
  The preservation of composition for the $\mathcal{V}$-functor $T$ is expressed
  by the commutativity of the top diagram in Figure~\ref{fig:5}, which can be deduced from the
  commutativity of the third diagram in Figure~\ref{fig:1} (expressing the monoidality of $\alpha$)
  by setting $X=\ul{A}(C,E)$, $X'=\ul{A}(D,G)$, $Y=\ul{A}(A,C)$,
  $Y'=\ul{A}(B,D)$, $A=A\diamond B$ and $A'=E\diamond G$.
  \begin{figure}
  \begin{equation}
    \xymatrix@C=.25in{
      \ul{A}(C,E)\otimes\ul{A}(D,G)\otimes\ul{A}(A,C)\otimes\ul{A}(B,D)
      \ar[r]^-{T\otimes T}
      \ar[d]_{1\otimes c\otimes 1}
      &
      \ul{A}(C\diamond D,E\diamond G)\otimes\ul(A\diamond B,C\diamond D)
      \ar[dd]^{\mathrm{comp}}
      \\
      \ul{A}(C,E)\otimes\ul{A}(A,C)\otimes \ul{A}(D,G)\otimes\ul{A}(B,D)
      \ar[d]_{\mathrm{comp}\otimes \mathrm{comp}}
      &
      \\
      \ul{A}(A,E)\otimes\ul{A}(B,G)\ar[r]^-T
      &
      \ul{A}(A\diamond B,E\diamond G)
    }
  \end{equation}
  \begin{equation}
    \xymatrix{
      \ul{A}(A,A)\otimes\ul{A}(B,B)\ar[r]^-{T}&
      \ul{A}(A\diamond B,A\diamond B)\\
      I\ar[u]^{\mathrm{id}\otimes\mathrm{id}}\ar[ur]_{\mathrm{id}}
    }
  \end{equation}
  \caption{$\mathcal{V}$-functor axioms for $T$.}\label{fig:5}
  \end{figure}
  The preservation of identities for the $\mathcal{V}$-functor $T$ is the
  commutativity of the bottom diagram in Figure~\ref{fig:5}, which once translated under the adjunction
  $(-*(A\diamond B))\dashv\ul{A}(A\diamond B,-)$, can be easily seen to hold by
  naturality of $\xi$.

  We can now use Theorem~\ref{mainthm} to complete the proof. We must show that
  the morphisms $T_{ABCD}$ and $\mathrm{id}_J\colon I\to\ul{A}(J,J)$ form a
  monoidal structure on functor
  $\ul{A}(-,-)\colon\mathcal{A}^\mathrm{op}\times\mathcal{A}\to\mathcal{V}$. This
  is precisely the case, by Proposition~\ref{param}, since $T_{ABCD}$ and
  $\mathrm{id}_J$ are the transpose of the opmonoidal structure of the action $*$
  under the parametrised adjunctions $(-*A)\dashv\ul{A}(A,-)$.
\end{proof}
In the case of a right closed monoidal category $\mathcal{V}$ acting on itself
via the tensor product, Theorem~\ref{thisthm} says that $\mathcal{V}$ is a
monoidal $\mathcal{V}$-category, provided that it is equipped with a braiding.
\begin{ex}
  \label{ex:2}
  If in Example~\ref{ex:1} the monoidal right closed category $\mathcal{V}$ is
  braided, then $[-,-]^\op$ is an opmonoidal action of $\mathcal{V}$ on
  $\mathcal{V}^{\mathrm{op}}$. The opmonoidal structure is given by morphisms of
  the form~\eqref{eq:19} but in $\mathcal{V}^{\mathrm{op}}$. These are the
  canonical isomorphism $I\cong [I,I]$ and the morphism $[X,Y]\otimes
  [Z,W]\to[X\otimes Z,Y\otimes W]$ that makes $[-,-]$ a monoidal functor.
\end{ex}

\begin{df}
  \label{df:2}
  Suppose given an opmonoidal action as in
  Definition~\ref{df:1} and suppose that the monoidal category $\mathcal{A}$ braided. We
  say that the opmonoidal action is \emph{braided} when the opmonoidal functor $*$ is
  so.
\end{df}
In more explicit terms, the opmonoidal action is braided when the following
diagram commutes, where both the braiding of $\mathcal{V}$ and $\mathcal{A}$ are
denoted by $c$.
\begin{equation}
  \label{eq:21}
  \xymatrix{
    (X\otimes Y)*(A\diamond B)\ar[r]^-{\xi}\ar[d]_{c*c}&
    (X*A)\diamond(Y*B)\ar[d]^{c}\\
    (Y\otimes X)*(B\diamond A)\ar[r]^-\xi&
    (Y*B)\diamond(X*A)
  }
\end{equation}
\begin{ex}
  \label{ex:3}
  Continuing with Example~\ref{ex:2}, the action of $\mathcal{V}$ on
  $\mathcal{V}^{\mathrm{op}}$ given by the internal hom is braided if
  $\mathcal{V}$ is symmetric. The diagram~\eqref{eq:21} in this case looks as
  follows (where we use that $c^{-1}=c$).
  \begin{equation}
    \label{eq:35}
    \xymatrix{
      [X,A]\otimes[Y,B]\ar[r]\ar[d]_{c}&
      [X\otimes Y,A\otimes B]\ar[d]^{[c,c]}\\
      [Y,B]\otimes[X,A]\ar[r]&
      [Y\otimes X,B\otimes A]
    }
  \end{equation}
\end{ex}
\begin{thm}\label{braidthm}
  Suppose that in Theorem~\ref{thisthm}
  $\mathcal{A}$ is braided. Then the opmonoidal action is braided if and only if
  the braiding of $\mathcal{A}$ is $\mathcal{V}$-natural. In this situation
  $\underline{\mathcal{A}}$ is braided.
\end{thm}
\begin{proof}
  The commutativity of~\eqref{eq:21} is equivalent to the
  commutativity of
  \begin{equation}
    \label{eq:22}
    \xymatrix{
      \underline{\mathcal{A}}(A,B)\otimes \underline{\mathcal{A}}(C,D)
      \ar[r]^-{T_{ACBD}}\ar[d]_c
      &
      \underline{\mathcal{A}}(A\diamond C,B\diamond
      D)\ar[d]^{\underline{\mathcal{A}}(c^{-1},c)}
      \\
      \underline{\mathcal{A}}(C,D)\otimes \underline{\mathcal{A}}(A,B)
      \ar[r]^-{TCADB}
      &
      \underline{\mathcal{A}}(C\diamond A,D\diamond B)
    }
  \end{equation}
  where we used the notation of the proof of Theorem~\ref{thisthm}. This is the
  condition that ${\mathcal{A}}(-,-)$ is braided monoidal, required by Theorem~\ref{mainthm}.
\end{proof}

\section{The universal measuring comonoid}\label{existmeascoal}

The notion of the universal measuring coalgebra
$P(A,B)$ over a field $k$ appeared in Sweedler's book~\cite{Sweedler}.
The elements of $P(A,B)$ can be thought of as generalised
maps from $A$ to $B$, and examples of this
point of view are given in~\cite{Batchelor}.
The natural isomorphism
that defines the object $P(A,B)$ is
\begin{equation}\label{this}
\Alg_k(A,\Hom_k(C,B))\cong\Coalg_k(C,P(A,B)).
\end{equation}
Note that the plain algebra morphisms $A\to B$
correspond to the group-like elements of $P(A,B)$.

Our aim in this section is to prove the existence of $P(A,B)$
in a broader context, identifying
the underlying categorical ideas.
In that direction, consider an arbitrary 
braided monoidal closed category $\ca{V}$.

We remind the reader that
in Section~\ref{background} we saw how
the internal hom induces a functor
$[-,-]\colon \Comon(\ca{V})^\mathrm{op}\times
\Mon(\ca{V})\to \Mon(\ca{V})$.

\begin{thm}\label{meascomonthm}
Suppose that $\ca{V}$ is a locally presentable braided monoidal
closed category. Then the
functor $[-,B]^{\mathrm{op}}:\Comon(\ca{V})\to
\Mon(\ca{V})^\mathrm{op}$ has a right adjoint
$P(-,B)$; ie there is a natural isomorphism
\begin{equation}\label{meascomon}
\Mon(\ca{V})(A,[C,B])\cong
\Comon(\ca{V})(C,P(A,B)).
\end{equation}
\end{thm}
\begin{proof}
By Proposition~\ref{moncomonadm}, the category
$\Comon(\ca{V})$ is locally presentable.
The diagram
\begin{equation}\label{Hcont}
  \xymatrix@C=1cm{
    \Comon(\ca{V})
    \ar[r]^-{[-,B]^{\mathrm{op}}}\ar[d]_U&
    \Mon(\ca{V})^{\mathrm{op}}\ar[d]^-V\\
    \ca{V}\ar[r]^-{[-,U(B)]^{\mathrm{op}}} &\ca{V}^\mathrm{op}
  }
\end{equation}
commutes, where $U$ and $V$ are the forgetful functors, and since
$[-,U(B)]^{\mathrm{op}}$ is cocontinuous, so is the composition
$[U-,U(B)]^{\mathrm{op}}$. Therefore, the the functor at the top of the diagram
is cocontinuous, since $V$ creates colimits. The
existence of the adjoint $P(-,B)$ now follows from the locally presentability of
$\Comon(\mathcal{V})$ (Remark~\ref{rmk:7}).
\end{proof}

The object $P(A,B)$ for monoids $A$ and $B$
is called the \emph{universal measuring comonoid}, and
the parametrised adjoint of $[-,-]^\mathrm{op}$, namely
\begin{equation}\label{defP}
P:\Mon(\ca{V})^\mathrm{op}\times\Mon(\ca{V})
\to\Comon(\ca{V})
\end{equation}
is called the \emph{Sweedler hom} in~\cite{AnelJoyal}.

\begin{cor}
  \label{cor:5}
  The functor $P(-,-)\colon\Mon(\mathcal{V})^{\mathrm{op}}\times
  \Mon(\mathcal{V})\to\Comon(\mathcal{V})$ is continuous in each variable.
\end{cor}
\begin{proof}
  By definition, $P(-,B)$ is a right adjoint, thus continuous. In the other
  variable, $P(A,-)$ preserves limits if and only if
  $\Comon(\mathcal{V})(C,P(A,-))$ does for all $C$, if and only if
  $\Mon(\mathcal{V})(A,[C,-])$ does, by Theorem~\ref{meascomonthm}; this last
  condition is clearly true.
\end{proof}

\begin{ex}
\label{ex:4}
Since the category $\Mod_k$ of modules over a commutative ring $k$ is
locally presentable symmetric monoidal closed,
we recover the classical situation: the existence, of each pair of $k$-algebras
$A$, $B$, of a \emph{universal measuring coalgebra} $P(A,B)$ with a natural
isomorphism~\eqref{meascomon}. See also~\cite[Prop.~4]{AdjAlgCoalg}.
\end{ex}
When $k$ is a field, $P(A,k)$ is usually denoted by $A^\circ$ and called the
\emph{finite} or \emph{Sweedler dual} of the $k$-algebra $A$. It can be
presented as the subspace of $A^*$
\[
A^{\circ}=\{\alpha \in A^*|\ker \alpha\textrm{ contains a cofinite ideal}\}
=\{\alpha\in A^*|\dim(A\rightharpoonup{}\alpha)<\infty\}
\]
where $\rightharpoonup$ denotes the left action of $A$ on $A^*$ given by
$(a\rightharpoonup \alpha)(x)=\alpha(xa)$.

\begin{ex}
  \label{ex:5}
  The locally presentable braided monoidal closed category $\mathbf{gVect}_{\mathbb{N}}$ of
  $\mathbb{N}$-graded vector spaces was described in Example~\ref{ex:14}.
  Monoids and comonoids in $\mathbf{gVect}_{\mathbb{N}}$ are, respectively,
  \emph{graded $k$-algebras} and \emph{graded $k$-coalgebras}.

  The full inclusion $(-)[0]\colon\mathbf{Vect}\to
  \mathbf{gVect}_{\mathbb{N}}$ that takes a space $V$ to the graded space $V$
  concentrated in degree $0$ is a left and a right adjoint to the functor
  that takes the homogeneous component of degree $0$. All three functors are
  strong monoidal. These adjunctions induce adjunctions on the respective
  categories of monoids as the ones depicted vertically in the following
  diagram. Given a commutative algebra $B$, there are adjunctions depicted
  horizontally.
  \begin{equation}
    \label{eq:25}
    \xymatrix@C=2.3cm@R=.9cm{
      \mathbf{gCoalg}
      \ar@<5pt>[r]^-{[-,B[0]]}\ar@<-5pt>@{<-}[r]_{P(-,B[0])}\ar@{}[r]|\bot
      \ar@<-5pt>@{<-}[d]_{(-)[0]}\ar@<5pt>[d]^{(-)_0}\ar@{}[d]|\dashv
      &
      \mathbf{gAlg}^{\mathrm{op}}
      \ar@<-5pt>@{<-}[d]_{(-)[0]}\ar@<5pt>[d]^{(-)_0}\ar@{}[d]|\dashv
      \\
      \mathbf{Coalg}
      \ar@<5pt>[r]^-{[-,B]}\ar@<-5pt>@{<-}[r]_-{P(-,B)}\ar@{}[r]|\bot
      &
      \mathbf{Alg}^{\mathrm{op}}
    }
  \end{equation}
  The square of left adjoints commute, since
  $[V,B[0]]_n=\prod_{i}[V_i,B[0]_{i+n}]$ is trivial unless $n=0$, in which case
  it is isomorphic to $[V_0,B]$. It follows that the square of right functors
  commute up to isomorphism, ie
  \begin{equation}
    \label{eq:62}
    P(A,B[0])_0\cong P(A_0,B).
  \end{equation}
\end{ex}
\begin{ex}
  \label{ex:10}
  Recall from Example~\ref{ex:13} the locally finitely presentable
  monoidal category $\mathbf{dgVect}_{\mathbb{Z}}$ 
  of chain complexes.
  Monoids and comonoids in $\mathbf{dgVect}_{\mathbb{Z}}$ are usually called
  \emph{dg algebras} and \emph{dg coalgebras}, and the categories they form
  $\mathbf{dgAlg}_{\mathbb{Z}}$ and $\mathbf{dgCoalg}_{\mathbb{Z}}$.
  The universal measuring comonoid of two dg algebras is the Sweedler hom
  considered in~\cite[\S 4.1.5]{AnelJoyal}.
\end{ex}

\section{Enrichment of monoids in comonoids}\label{enrichmonscomons}

Now that we have established the existence of the universal
measuring comonoid $P(A,B)$ under certain hypotheses,
we may combine this construction with the theory
of actions of monoidal categories
of Section~\ref{actions} in order to
exhibit an enrichment of monoids over comonoids. In this section,
$\mathcal{V}$ will denote a locally presentable braided monoidal closed
category, with braiding $c$. Recall that the internal hom functor $[-,-]$ is
monoidal, and that the monoidal category of comonoids $\Comon(\mathcal{V})$
is symmetric when $\ca{V}$ is.

\begin{lem}\label{actionH}
The functor
\([-,-]^\op\colon\Comon(\ca{V})\times
\Mon(\ca{V})^\mathrm{op}\to \Mon(\ca{V})^\mathrm{op}
\) (\ref{MonHom})
is an action of the monoidal
category $\Comon(\ca{V})$ on
$\Mon(\ca{V})^{\mathrm{op}}$.
If the braiding of $\mathcal{V}$ is a symmetry, then
this is a braided opmonoidal action of the symmetric monoidal category of comonoids.
\end{lem}
\begin{proof}
  The functor of the statement is obtained by taking comonoid categories on
  $[-,-]^\op\colon\mathcal{V}\times\mathcal{V}^{\mathrm{op}}\to\mathcal{V}^{\mathrm{op}}$. The
  latter is an opmonoidal left action of the braided monoidal $\mathcal{V}$
  (Example~\ref{ex:2}) upon which $\Mon[-,-]^\op$ inherits the structure of a
  left action.

  If $\mathcal{V}$ is symmetric, the internal hom is a braided opmonoidal action
  of $\mathcal{V}$ on $\mathcal{V}^{\mathrm{op}}$, by Example~\ref{ex:3}. Taking
  categories of monoids, we obtain a braided opmonoidal action of the
  symmetric category of comonoids on $\Mon(\mathcal{V})^{\mathrm{op}}$.
\end{proof}

We can now apply Proposition~\ref{actionenrich}, Theorem~\ref{thisthm} and
Theorem~\ref{braidthm} to $\ca{C}=\Comon(\ca{V})$ and $\ca{A}=\Mon(\ca{V})^\op$. The
functor $[-,B]^\op$ of Lemma~\ref{actionH} has a right adjoint
$P(-,B)$ by Theorem~\ref{meascomonthm}.

\begin{thm}
  \label{thm:1}
  Let $\mathcal{V}$ be a locally presentable braided monoidal closed
  category. Then:
  \begin{enumerate*}
  \item \label{item:1}
    The category $\Mon(\mathcal V)^\mathrm{op}$ is enriched
    in $\Comon(\ca{V})$, with enriched hom objects
    $ \Mon(\ca{V})^\mathrm{op}(A,B)=P(B,A)$.
    \end{enumerate*}
    
  If $\mathcal{V}$ is moreover symmetric, then:
  \begin{enumerate*}[resume]
  \item \label{item:2}
    $\Mon(\mathcal{V})^{\mathrm{op}}$ also carries a structure of a symmetric
    monoidal $\Comon(\mathcal{V})$-enriched category.
  \item \label{item:3} $\Mon(\ca{V})$ is a symmetric monoidal $\Comon(\ca{V})$-category too, with
    hom objects $\Mon(\ca{V})(A,B)=P(A,B)$.
\end{enumerate*}
\end{thm}

Assume for the rest of the section that the braiding of $\mathcal{V}$ is a symmetry.
By Lemma~\ref{Hsymmetriclax}, the functor of Lemma~\ref{actionH} is a symmetric
opmonoidal functor. Hence, by Proposition~\ref{param}
we get the following result.

\begin{cor}\label{Plaxmon}
The Sweedler hom functor
$P:\Mon(\ca{V})^\mathrm{op}\times\Mon(\ca{V})
\to\Comon(\ca{V})$
is a braided monoidal functor.
\end{cor}

Finally, since $P$ is a monoidal
functor, it induces a functor
\begin{equation}
\Mon P:\Bimon(\ca{V})^\mathrm{op}
\times \mathbf{CommMon}(\ca{V})\to
\Bimon(\ca{V})
\end{equation}
where $\mathbf{CommMon}(\ca{V})=\Mon(\Mon(\ca{V}))$
is the category of commutative monoids, and of course
$\Mon(\Comon(\ca{V}))=\Bimon(\ca{V})$ is the category of
bimonoids.
This is still a braided monoidal functor
by Remark~\ref{rmk:2} and so, since
$\Mon(\Bimon(\ca{V}))= \mathbf{CommBimon}(\ca{V})$,
we get the following result.

\begin{cor} \label{dualbimonoid}
Suppose $\ca{V}$ is a locally presentable symmetric monoidal closed category.
If $B$ is a (cocommutative) bimonoid
and $A$ a commutative monoid, then $P(B,A)$
has a canonical structure of a (commutative) bimonoid.
In particular, the
finite dual $B^\circ$ is a (commutative) bimonoid.
\end{cor}
Note that the second part is also proved, in a much different way,
in \cite{Onbimeasurings} for the
case $\ca{V}=\Mod_R$.

\begin{rmk}
    \label{rmk:6}
  When $(B,\iota,\mu)$ is a commutative monoid in
  $\mathcal{V}$, $[-,B]\colon\Comon(\mathcal{V})\to\Mon(\mathcal{V})^{\mathrm{op}}$ is
  an opmonoidal functor; its structure is given by the morphisms in
  $\mathcal{V}$
  \begin{equation}
    \label{eq:79}
    \chi\colon [A,B]\otimes[A',B]\longrightarrow[A\otimes A',B\otimes B]\xrightarrow{[1,\mu]}
    [A\otimes A',B] \qquad
    \chi_0\colon I\cong[I,I]\xrightarrow{[I,\iota]}[I,B].
  \end{equation}
  If we denote the counit of the adjunction $[-,B]\dashv P(-,B)$ by the morphism in
  $\Mon(\mathcal{V})$
  \begin{equation}
    \label{eq:80}
    \eta_A\colon A\longrightarrow[P(A,B),B],
  \end{equation}
  then the monoidal structure on the right adjoint $P(-,B)$ is given by the morphisms
  \begin{equation}
    \label{eq:81}
    \psi\colon P(A,B)\otimes P(A',B)\longrightarrow P(A\otimes A',B)\qquad
    \psi_0\colon I\to P(I,B)
  \end{equation}
  defined as the unique ones that make the following diagrams
  commute.
    \begin{equation}
      \label{eq:82}
      \xymatrixcolsep{.3cm}
      \diagram
      A\otimes A'\ar[r]^-{\eta_{A\otimes A'}}\ar[d]_{\eta_A\otimes\eta_{A'}}&
      [P(A\otimes A',B),B]\ar[d]^{[\psi,B]}\\
      [P(A,B),B]\otimes[P(A',B),B]\ar[r]^-{\chi}&
      [P(A,B)\otimes P(A',B),B]
      \enddiagram
      \diagram
      I\ar[r]^-{\eta_I}\ar[d]_\iota&
      [P(I,B),B]\ar[d]|{[\psi_0,B]}\\
      B\ar[r]^-{\sim}&
      [I,B]
      \enddiagram
    \end{equation}
\end{rmk}

\section{Comodules of universal measuring coalgebras}
\label{sec:calc-meas-coalg}

Having established the enrichment of the category
of monoids in the category of comonoids via the universal measuring comonoid,
in this section we study these objects primarily from the point of view
of their comodules or {corepresentations}, exhibiting further properties along the way.

\subsection{The finite dual as a subobject of a cofree comonoid}
\label{sec:finite-dual-as}

If $\mathcal{V}$ is a locally presentable monoidal category, it is not hard to
show that free monoids exist in $\mathcal{V}$, and then, $\Mon(\mathcal{V})$ becomes monadic
over $\mathcal{V}$. We say only a few words about the proof. Since both $\Mon(\mathcal{V})$ and
$\mathcal{V}$ are locally presentable (Proposition~\ref{moncomonadm}),
it suffices to know that the
forgetful functor from the former to the latter is continuous and accessible (by Remark~\ref{rmk:7});
see also~\cite[Thm.~5.32]{Kelly} for a more
general result. The fact that the forgetful functor preserves
$\kappa$-filtered colimits, for some regular cardinal $\kappa$, can be easily
verified using the fact that the tensor product of $\mathcal V$ does so. This concludes our
sketch of a proof.

Easier still is to prove the fact that cofree comonoids exist in any locally
presentable monoidal category $\mathcal{V}$; for, the forgetful functor from
$\Comon(\mathcal{V})$ to $\mathcal{V}$ is cocontinuous, and thus a left adjoint
again by Remark~\ref{rmk:7}.

We shall denote the free monoid on $X\in \mathcal{V}$ by $T(X)$. As the notation
suggests, the free monoid in the category of $k$-modules, for a commutative ring
$k$, is the tensor algebra. The cofree comonoid on $X$ we shall denote by $S(X)$.

In this section $\mathcal{V}$ will be a locally presentable \emph{braided}
monoidal closed category.

\begin{lem}
  \label{l:2}
  For any monoid $B$ and any object $X\in \mathcal V$, $P(T(X),B)\cong
  S([X,B])$. In particular, $T(X)^\circ\cong S([X,I])$.
\end{lem}
\begin{proof}
  Consider the commutative diagram~\eqref{Hcont}. All four functors have a right
  adjoint, thus the diagram formed by the right adjoints commutes up to natural
  isomorphism, whose component at $X$ has domain and codomain those of the
  statement.
\end{proof}
Let $V$ be the forgetful $\Mon(\ca{V})\to\ca{V}$.
The functor $P(-,B)$ sends colimits in $\Mon(\mathcal{V})$ to limits in
$\Comon(\mathcal{V})$ by adjointness. In particular,
it sends the canonical
diagram $T^2V(A)\rightrightarrows TV(A)\to A$ that exhibits a monoid $A$ as
coequaliser of free monoids, into an equaliser
\begin{equation}
  \label{eq:23}
  P(A,B)\rightarrowtail{}S[V(A),V (B)]\rightrightarrows S[VT(A),V(B)].
\end{equation}
In the case when $\mathcal{V}$ is the category of $k$-vector spaces and $B=k$,
this equaliser exhibits $A^\circ$ as a subcoalgebra of the cofree coalgebra on
$A^*$. 
Composing with the counit $S\Rightarrow 1$ of the cofree coalgebra comonad,
we obtain a morphism
\begin{equation}
A^\circ \longrightarrow A^*\label{eq:3}
\end{equation}
that is the classical injection of the finite dual into the dual
space~\cite{Sweedler}.

\subsection{Coendomorphism comonoids}
\label{sec:coend-comon}

Recall from the background Section \ref{background} the notion of a dual object.
The \emph{coendomorphism comonoid} of an object $X$ with left dual $X^\vee$ is the
object $X^\vee\otimes X$, with comultiplication $X\otimes\coev \otimes X$ and counit
$\ev$. We shall denote it by $\operatorname{coend}(X)$. These comonoids are
useful to us because $C$-comodule structures $X\to X\otimes C$ are in bijection
with comonoid morphisms $\operatorname{coend}(X)\to C$. In particular,
the coendomorphism coalgebras offer a reinterpretation of the
so-called \emph{fundamental theorem of coalgebras} below.

Recall that a set of objects $\mathcal{G}\subset \operatorname{ob}\mathcal{C}$
is strongly generating if, the functors
$\{\mathcal{C}(G,-)\colon\mathcal{C}\to\mathbf{Set}\}_{G\in\mathcal{G}}$ are
jointly conservative, ie if a morphism $f$ is invertible whenever $\mathcal{C}(G,f)$
is a invertible for all $G\in\mathcal{G}$. See~\cite[\S 3.6]{Kelly}.

\begin{lem}
  \label{l:4}
  When $\mathcal{V}$ is the category of
  $k$-vector spaces, the family of coendomorphism coalgebras
  $\{\operatorname{coend}(k^n)\}_{n\geq 1}$ is strongly generating in
  $\mathbf{Coalg}_k$.
\end{lem}
\begin{proof}
  If $X$ is a finite-dimensional $C$-comodule, the image of the associated
  $X^*\otimes X\to C$ is called the \emph{coefficient space} or \emph{coalgebra
    of coefficients} of $X$, denoted by $\mathrm{cf}(X)$. It is the smallest
  subcoalgebra of $C$ for which $X$ is a comodule; see~\cite[\S 1.2]{MR0412221}.

  By the fundamental theorem of coalgebras~\cite[Thm.~2.2.1]{Sweedler}, $C$ is union of finite-dimensional
  subcoalgebras. It is not hard to see that, if $D\subset C$ is a finite
  dimensional subcoalgebra regarded as a $C$-comodule, then $\mathrm{cf}(D)=D$
  (for,
  evaluating $D^*\otimes D\to C$ on $\varepsilon_{D}\otimes d$ gives back
  $d$). Therefore, the morphism of coalgebras $\sum_D\operatorname{coend}(D)\to
  C$ induced by the morphisms $\operatorname{coend}(D)\to C$, for each finite
  dimensional subcoalgebra $D\subset C$, is surjective. Hence, the following
  morphism is surjective (where $S\cdot E$, for a set $S$ and a coalgebra $E$,
  denotes the copower, ie the coproduct of $S$-copies of $E$).
  \begin{equation}
    \label{eq:30}
    \sum_n\mathbf{Coalg}_k(\operatorname{coend}(k^n),C)\cdot
    \operatorname{coend}(k^n)\longrightarrow C
  \end{equation}
  In particular, \eqref{eq:30}~is an
  extremal epimorphism, ie it does not factor through any non-trivial subobject
  of $C$; for more information see the paragraph previous to~\cite[Prop.~4.6]{MR0240161},
  or~\cite[\S 8.7]{Kelly:StructuresFiniteLimits}. This is equivalent to saying that the
  coalgebras $\operatorname{coend}(k^n)$ form a strong generator
  (see~\cite[\S 8.7]{Kelly:StructuresFiniteLimits}).
\end{proof}

\subsection{Comodules over the universal measuring coalgebra}
\label{sec:comod-over-univ}
Recall from Section \ref{Kleislicats} the Kleisli construction for an enriched monad
on a braided monoidal closed $\ca{V}$. We will be interested in the enriched monad
$T=(-\otimes B)$ induced by tensoring with a monoid $B$; in this case we will
abbreviate the categories of Kleisli and of Eilenberg-Moore algebras by $\mathcal{V}_B$ and
$\mathcal{V}^B$. The former always has tensor products by objects of
$\mathcal{V}$ (in the sense of~\cite[\S 3.7]{Kelly}),
since the universal $\mathcal{V}\to\mathcal{V}_B$ is a left
adjoint; the tensor product of $X\in\mathcal{V}_B$ by $Z\in\mathcal{V}$ is
$Z\otimes X$. As is always the case, the base category $\mathcal{V}$ acts on
$\mathcal{V}_B$ on the left by tensor products.
The $\mathcal{V}$-monad $(A\otimes\mathbin{-})$ extends to a
$\mathcal{V}$-monad on $\mathcal{V}_B$ and lifts to a $\mathcal{V}$-monad on
$\mathcal{V}^B$ thanks to the isomorphism $(A\otimes X)\otimes B\cong
A\otimes(X\otimes B)$.

\begin{prop}
  \label{prop:7}
  Let $A$, $B$ be two monoids in the locally presentable braided monoidal closed
  category $\mathcal{V}$ and $X$ an object. There is a bijection between:
  \begin{enumerate}
  \item \label{item:16} Algebra structures on $X$, for the monad
    $(A\otimes\mathbin{-})$ on $\mathcal{V}_B$.
  \item \label{item:13} Monoid morphisms $A\to\mathcal{V}_{B}(X,X)=[X,X\otimes B]$.
  \item \label{item:15} Algebra structures on $X\otimes B$, for the monad
    $(A\otimes\mathbin{-})$ on $\mathcal{V}^B$.
  \item \label{item:14} Monoid morphisms $A\to\mathcal{V}^{B}(X\otimes B,X\otimes B)$.
  \end{enumerate}
  If $X$ has a dual, then the above data are equivalent to:
  \begin{enumerate}[resume]
  \item \label{item:6} Right $P(A,B)$-comodule structures on $X$.
  \end{enumerate}
\end{prop}
It may be instructive to spell out the properties that a morphism
$A\otimes X\to X$ has to satisfy in order to be an algebra structure on
$X\in\mathcal{V}_B$ for the monad $(A\otimes\mathbin{-})$, as in
the item~\ref{item:16} of the above proposition. It is a morphism $\psi\colon
A\otimes X\to X\otimes B$ in $\mathcal{V}$ that makes the following pair of
diagrams commute.
\begin{equation}
  \label{eq:26}
  \xymatrix{
    A\otimes A\otimes X\ar[r]^{1\otimes\psi}\ar[d]_{\mu\otimes1}&
    A\otimes X\otimes B\ar[r]^-{\psi\otimes 1} &
    X\otimes B\otimes B\ar[d]^{1\otimes \mu}\\
    A\otimes X\ar[rr]^-\psi&&
    X\otimes B
  }
  \qquad
  \xymatrix{
    X\ar[d]_{\eta\otimes 1}\ar@{=}[r]&X\ar[d]^{1\otimes \eta}\\
    A\otimes X\ar[r]^-\psi&X\otimes B
  }
\end{equation}
\begin{proof}
  Morphisms $\xi\colon A\otimes X\to X$ in $\mathcal{V}_B$ are in bijection with
  morphisms $\hat\xi\colon A\to\mathcal{V}_T(X,X)$ in $\mathcal{V}$, by the
  universal property of the tensor product with objects of $\mathcal{V}$. Under
  this correspondence, $\xi$ is an an algebra structure for $X$ if and only if
  $\hat\xi$ is a monoid morphism in $\mathcal{V}$, where the multiplication in
  its codomain is composition. This proves the equivalence of
  (\ref{item:16})~and~(\ref{item:13}). The equivalence between
  (\ref{item:15})~and~(\ref{item:14}) holds for precisely the same reason, while
  the equivalence of (\ref{item:13})~and~(\ref{item:14}) is a consequence of the
  full and faithful comparison $\mathcal{V}$-functor
  $\mathcal{V}_B\rightarrowtail\mathcal{V}^B$.

  If $X$ has a (left) dual $X^\vee$, the isomorphism $[X,X\otimes
  B]\cong[X^\vee\otimes X,B]$ becomes an isomorphism of monoids when the domain
  has the composition of $\mathcal{V}_B$ as multiplication and the codomain has
  the convolution multiplication induced by the comonoid
  $\operatorname{coend}(X)=X^\vee\otimes X$ as in Section \ref{sec:coend-comon},
  and the monoid $B$. Thus, a monoid morphism as
  in~(\ref{item:13}) can equally be given by a monoid morphism
  $A\to[\operatorname{coend}(X),B]$, and therefore by a comonoid morphism
  $\operatorname{coend}(X)\to P(A,B)$. This corresponds to a morphism $X\to
  X\otimes P(A,B)$ satisfying the comodule axioms.
\end{proof}

\begin{df}
  \label{df:3}
  Given two monoids $A$ and $B$, the category $\avb$ has objects
  pairs $(X,\psi)$, where $X\in\mathcal{V}$ and $\psi\colon A\otimes X\to
  X\otimes B$ satisfies the two axioms depicted in the previous paragraph; it has
  morphisms $(X,\psi)\to(X',\psi')$ morphisms $f\colon X\to X'$ in $\mathcal{V}$
  that satisfy $(f\otimes B)\cdot\psi=\psi'\cdot(A\otimes f)$. Composition and
  identities are the obvious ones, so there is a faithful forgetful functor
  $\avb\to\mathcal{V}$.
\end{df}
The category just defined fits in the following pullback diagram, where
$(\mathcal{V}_B)^{(A\otimes\mathbin{-})}$ is the category of algebras of the
monad $(A\otimes\mathbin{-})$ on $\mathcal{V}_B$, and the bottom arrow is the
universal Kleisli functor.
\begin{equation}
  \label{eq:24}
  \diagram
    \avb\ar[r]\ar[d]&
    (\mathcal{V}_B)^{(A\otimes\mathbin{-})}\ar[d]\\
    \mathcal{V}\ar[r]&
    \mathcal{V}_B
  \enddiagram
\end{equation}
If we recall the notion of a dualizable object from the background
Section \ref{background}, we obtain the following result.
\begin{cor}
  \label{cor:3}
  There is an isomorphism between the categories of dualizable right
  $P(A,B)$-comodules and that of dualizable objects of $\avb$;
  furthermore, the isomorphism commutes with the respective forgetful functors
  into $\mathcal{V}$.
\end{cor}
The isomorphism of the previous corollary is given on objects by
Proposition~\ref{prop:7}. The rest of the details are left to the
reader. When $B=I$ we have:

\begin{cor}
  \label{cor:2}
  For a monoid $A$ in $\mathcal{V}$, the category of dualizable right
  $A^\circ$-comodules is isomorphic to the category of dualizable left
  $A$-modules.
\end{cor}
\begin{proof}
  Setting $B=I$ in Proposition~\ref{prop:7}, the Kleisli $\mathcal{V}$-category
  $\mathcal{V}_B$ becomes just $\mathcal{V}$, and the data in the
  item~(\ref{item:16}) of the said proposition just an $A$-module structure on
  $X$.
\end{proof}

In the example when $\mathcal{V}$ is the category of vector spaces,
Corollary~\ref{cor:3} gives an alternative description of the category of finite%
-dimensional right $P(A,B)$-comodules, for any pair of algebras $A$, $B$.
\begin{cor}
  \label{cor:4}
  If $A$, $B$ are algebras over a field $k$, then
  \begin{equation}
    \label{eq:28}
    P(A,B)\cong\int^{(X,\psi)}X^*\otimes X
  \end{equation}
  where $(X,\psi)$ runs over all the objects of $\avb$ with
  $\dim_kX<\infty$.
\end{cor}
\begin{proof}
  The forgetful functor $(\avb)_d\to\mathcal{V}$ from the category of
  dualizable objects of $\avb$ is, up to composing with an
  isomorphism, the forgetful functor from the category of dualizable
  $P(A,B)$-comodules. Then, the coalgebra can be reconstructed by the
  coend~\eqref{eq:28}; the ideas behind this reconstruction go back
  to~\cite{Saavedra:Tannakian}, but see for example~\cite{MR1623637} for a paper
  where coends are explicitly used.
\end{proof}
The corollary above holds for more general categories $\mathcal{V}$, as shown
in~\cite{McCrudden:Maschkean}, but we do not pursue that point.

\begin{cor}
  \label{cor:7}
  There is a bijection between right $P(A,B)$-comodule structures on $k^n$ and
  algebra morphisms $A\to\operatorname{M}_{n\times n}(B)$. There is a bijection
  between isomorphism classes of $n$-dimensional $P(A,B)$-comodules and the
  quotient of the set $\Alg_k(A,\operatorname{M}_{n\times n}(B))$ by the
  action of $\operatorname{GL}_n(k)$ on $\operatorname{M}_{n\times n}(B)$ by
  conjugation.
\end{cor}
\begin{proof}
  The result is an easy consequence of Proposition~\ref{prop:7}. The canonical
  isomorphism between $[k^n,k^n\otimes B]$ and $B^{n\times n}$ is compatible
  with the Kleisli multipli\-cation on the former and the matrix multiplication on
  the latter. Given two algebra morphisms $\sigma,\tau\colon A\to[k^n,k^n\otimes
  B]$, an invertible matrix $M\in \operatorname{GL}_n(k)$ represents an isomorphism
  of $P(A,B)$-comodules if and only if $(M\otimes B)\cdot \sigma(a)=\tau(a)\cdot
  M$, which in terms of $\operatorname{M}_{n\times n}(B)$ means that $M\sigma(a)=\tau(a)M$, for
  all $a\in A$.
\end{proof}
The following examples for $\ca{V}=\Vect_k$ provide
applications of the measuring coalgebra corepresentations point of view.
\begin{ex}
  \label{ex:6}
  Given a $k$-algebra $A$, there is an isomorphism of algebras $A^\circ\cong k$
  if and only if $A$ satisfies:
  \begin{enumerate}
  \item it has an augmentation $A\to k$, ie $k$ is an $A$-module;
  \item all the finite-dimensional modules are direct sums of the module $k$.
  \end{enumerate}
  This is a consequence of Corollary~\ref{cor:3}. For, $A^\circ\cong k$ if and only
  if the forgetful functor from the category of finite-dimensional
  $A^\circ$-comodules into the category of finite dimensional vector spaces is
  an isomorphism. But this category is isomorphic to the category of
  finite-dimensional $A$-modules.

  An example is the group algebra $k[G]$ for a infinite simple group of cardinality
  larger than that of the field $k$; any finite-dimensional representation of
  $k[G]$ is given by a group morphism $G\to\mathrm{Aut}_k(k^n)$, which cannot be
  injective by the cardinality assumptions, thus it must be trivial by
  simplicity of $G$. An example of such a group $G$ is $\mathrm{PSL}(2,K)$ for an
  infinite field $k\subset K$ of cardinality larger than that of $k$. This
  example was introduced in~\cite[Lemma~2.7]{MR860387}.
\end{ex}
\begin{ex}
  \label{ex:9}
  The coalgebra $A^\circ$ can be zero, as pointed out
  in~\cite[p.~114]{Sweedler}, for example, if $A$ is a infinite-dimensional
  division $k$-algebra. It can be instructive to deduce this from the universal
  property of the finite dual. The set $\mathbf{Alg}_k(A,C^*)$
  is empty for all non-zero finite-dimensional coalgebras $C$. Therefore,
  $\mathbf{Coalg}_k(C,A^\circ)$ has this same property, and the functions
  \begin{equation}
    \label{eq:56}
    \mathbf{Coalg}_k(C,0)\longrightarrow\mathbf{Coalg}_k(C,A^\circ)
  \end{equation}
  induced by the unique morphism of coalgebras $0\to A^\circ$ are isomorphisms,
  for $C$ of finite dimension. We conclude that $0\to A^\circ$ is an isomorphism,
  by Lemma~\ref{l:4}.
\end{ex}
\begin{ex}
  \label{ex:15}
  If $B$ has an augmentation $\varepsilon\colon B\to k$, there is an induced
  coalgebra morphism $P(A,\varepsilon)\colon P(A,B)\to P(A,k)=A^\circ$. In these
  circumstances, the equality $P(A,\varepsilon)\cdot
  P(A,\iota)=P(A,\varepsilon\cdot\iota)=P(A,1_k)=1_{A^\circ}$, induces functors
  on the categories of comodules
  \begin{equation}
    \label{eq:77}
    1=
    \big(
    \Comod(A^\circ)\xrightarrow{P(A,\iota)_*}\Comod(P(A,B))
    \xrightarrow{P(A,\varepsilon)_*} \Comod(A^\circ)
    \big)
  \end{equation}
  that exhibit the category of $A^\circ$-comodules as a retract of that of
  $P(A,B)$-comodules. These functors are given by corestriction of scalars, so
  they commute with the respective forgetful functors into $\mathbf{Vect}_k$,
  and are conservative.

  An $A^\circ$-comodule $X$ is simple if and
  only if $P(A,\iota)_*(X)$ is a simple $P(A,B)$-comodule. The proof of this
  claim is elementary. Both functors in~\eqref{eq:77} preserve monomorphisms and
  are conservative, so
  they induce a retraction
  \begin{equation}
    1=\big(
    \operatorname{Sub}(X)\longrightarrow
    \operatorname{Sub}(P(A,\iota)_*(X))\longrightarrow\operatorname{Sub}(X)
    \big)
    \label{eq:78}
  \end{equation}
  where both functions reflect equalities of comparable subobjects (ie if
  $S\subseteq T$ are sent to the same subobject, then $S=T$). Therefore,
  $\operatorname{Sub}(X)$ has only bottom and top element if and only if
  $\operatorname{Sub}(P(A,\iota)_*(X))$ satisfies the same property.

  As a consequence, $P(A,B)$ has simple comodules of dimension
  $n$ if $A$ has simple modules of dimension $n$.
  For example, if $B$ is augmented (eg $B=k[G]$ for a monoid $G$), then
  $P(\operatorname{M}_{n\times n}(k),B)$ has simple
  comodules of dimension $n$.

  Another example is $P(U(\mathfrak{sl}(2,\mathbb{C})),B)$, which we show to be
  infinite-dimensional. By the above comments, this coalgebra has
  $P(U(\mathfrak{sl}(2,\mathbb{C})),k)=U(\mathfrak{sl}(2,\mathbb{C}))^\circ$ as
  a retraction. The finite-dimensional comodules over the latter coalgebra can
  be identified with finite-dimensional
  $\mathfrak{sl}(2,\mathbb{C})$-representations; in particular,
  $U(\mathfrak{sl}(2,\mathbb{C}))^\circ$ has simple comodules of all dimensions,
  and therefore it is an infinite-dimensional coalgebra. This last claim can be
  deduced from~\cite[Cor.~4.5]{MR2394705} which exhibits a bijection between
  isomorphism classes of simple comodules and simple subcoalgebras of a given
  coalgebra; therefore $U(\mathfrak{sl}(2,\mathbb{C}))^\circ$ has infinitely
  many non-isomorphic simple subcoalgebras, and hence it is
  infinite-dimensional.
\end{ex}


\subsection{Tambara's coendomorphism algebra}
\label{sec:tamb-coend-algebra}
D.~Tambara introduced in~\cite{MR1071429} an algebra $a(A,B)$ for each pair of
algebras $A$, $B$ over a field $k$, called the \emph{coendomorphism
 algebra}, with the property that there is a bijection
\begin{equation}
  \label{eq:4}
  \mathbf{Alg}_k(a(A,B),A')\cong\mathbf{Alg}_k(B,A\otimes A')
\end{equation}
natural in $C$. The $a(A,B)$-modules are described in \S 2 of op.~cit. in a
way similar to our Proposition~\ref{prop:7}. More precisely, finite
dimensional $a(A,B)$-modules can be identified with finite dimensional
$P(B,A)$-comodules.
\begin{prop}
  \label{prop:2}
  For all algebras $A$ and $B$ over a field $k$, there is a canonical
  isomorphism $a(A,B)^\circ\cong P(B,A)$.
\end{prop}
\begin{proof}
  There is a function, natural in $C\in\mathbf{Coalg}_k$, which, by Yoneda's
  lemma is
  induced by a unique morphism of coalgebras $f\colon a(A,B)^\circ\to P(B,A)$.
  \begin{multline}
    \label{eq:5}
    \mathbf{Coalg}_k(C,a(A,B)^\circ)\cong \mathbf{Alg}_k(a(A,B),C^*)\cong\\
    \cong
    \mathbf{Alg}_k(B,A\otimes C^*)
    \longrightarrow
    \mathbf{Alg}_k(B,[C,A])\cong \\
    \cong\mathbf{Coalg}_k(C,P(B,A))
  \end{multline}
  Here we used that the canonical inclusion $A\otimes C^*\hookrightarrow{}[C,A]$
  is a morphism of algebras, as it can be readily verified. We can now use that
  this inclusion is an isomorphism if $C$ is finite-dimensional, so the
  function~\eqref{eq:5} is an isomorphism in that case. Using the fact that
  finite-dimensional coalgebras are strong generating (Lemma~\ref{l:4}), we
  deduce that $f$ is an isomorphism.
\end{proof}

\section{Monoidal structures}
\label{sec:monoidal-structures}
There are two natural ways in which the universal measuring comonoid acquires a
bimonoid structure, and two ways in which the category of dualizable comodules acquires a
monoidal structure. In this section we take these two ways in turn, and give an
explicit description of the associated monoidal structures.

First,
  we have seen in Corollary~\ref{dualbimonoid} that, when $\mathcal{V}$ is a
  symmetric monoidal closed category,
  the comonoid $P(A,B)$ has a bimonoid structure if $A$ is a bimonoid and $B$ a
  commutative monoid. Then, the category \(\Comod_d(P(A,B))\) of dualizable
  right $P(A,B)$-comodules has a monoidal structure, that can be transferred to
  the equivalent category \((\avb)_d\) of dualizable objects of $\avb$, see
  Definition \ref{df:3}. The resulting monoidal structure on
  $(\avb)_d$ is given in the following way.
  \begin{cor}
    \label{cor:1}
    Given a bimonoid $(A,\mu,\iota)$ and a commutative monoid $(B,\Delta,\epsilon)$
    in a symmetric monoidal closed locally presentable category $\mathcal{V}$,
    the isomorphism between $\Comod_d(P(A,B))$ and $(\avb)_d$ is a monoidal isomorphism when we equip:
    \begin{enumerate}
    \item \label{item:9} $\Comod_d(P(A,B))$ with the monoidal structure associated
      to the induced bimonoid structure on $P(A,B)$.
    \item \label{item:10} $(\avb)_d$ with the monoidal structure defined as follows:
      \begin{enumerate}
      \item \label{item:22} if $(X,\varphi)$ and $(Y,\psi)$ are two objects, their tensor
        product is $X\otimes Y$ equipped with
        \begin{equation}
          \label{eq:51}
          AXY\xrightarrow{\Delta XY}AAXY\xrightarrow{AcY}AXAY
          \xrightarrow{\varphi\psi} XBYB\xrightarrow{XcB}XYBB
          \xrightarrow{XY\mu} XYB
          \notag
        \end{equation}
        where $\otimes$ is omitted.
      \item \label{item:33} The monoidal unit is
        $I$ equipped with
        \begin{equation}
          \label{eq:52}
          A\otimes I\xrightarrow{\varepsilon\otimes I}I\otimes
          I\xrightarrow{I\otimes\iota} I\otimes B.\notag
        \end{equation}
      \item \label{item:36} The forgetful functor $(\avb)_d\to\mathcal{V}$ is strict monoidal.
      \end{enumerate}
    \end{enumerate}
  \end{cor}
The second way in which $P(A,B)$ has a bimonoid structure is when $A=B$. The
multiplication $P(A,A)^{\otimes 2}\to P(A,A)$ is the morphism of comonoids that
corresponds to
\begin{equation}
  \label{eq:6}
  A\xrightarrow{\eta}[P(A,A),A]\xrightarrow{[1,\eta]}\big[P(A,A),[P(A,A),A]\big]\cong[P(A,A)^{\otimes
  2},A]
\end{equation}
where $\eta$ denotes the unit of the adjunction between $P(-,A)$ and
$[-,A]$. The unit $I\to P(A,A)$ is the morphism of coalgebras that corresponds
to the identity $A\to A$.
  \begin{cor}
    \label{cor:10}
    Given a monoid $A$ in a locally presentable monoidal category $\mathcal{V}$,
    the isomorphism between $\Comod_d(P(A,A))$ and $(\avb)_d$ becomes an
    isomorphism of monoidal categories when we equip:
    \begin{enumerate}
    \item \label{item:37}
      $\Comod_d(P(A,A))$ with the monoidal structure associated to the
      bimonoid structure on $P(A,A)$.
    \item \label{item:38} $(\avb)_d$ with the monoidal structure defined as
      follows:
      \begin{enumerate}
      \item if $(X,\varphi)$ and $(Y,\psi)$ are two of its objects, their tensor
        product is $X\otimes Y$ equipped with
        \begin{equation}
          \label{eq:53}
          A\otimes X\otimes Y\xrightarrow{\varphi \otimes Y }X\otimes A\otimes Y
          \xrightarrow{X\otimes \psi} X\otimes Y\otimes A.\notag
        \end{equation}
      \item
        The monoidal unit is $I$ equipped with $A\otimes I\cong A\cong I\otimes
        A$.
      \item The forgetful functor $(\avb)_d\to\mathcal{V}$ is strict monoidal.
      \end{enumerate}
    \end{enumerate}
  \end{cor}

\section{Universal measuring coalgebras and cocommutativity}
\label{sec:univ-meas-coalg-1}
We now return to the more general case of monoids and comonoids in a symmetric
monoidal closed category $\mathcal{V}$. Recall from Section \ref{sec:mono-comon-bimon}
the opposite (co)monoids.
\begin{lem}
  \label{l:5}
  If $\mathcal{V}$ is symmetric monoidal closed, there is a natural isomorphism
  $P(A,B)^{\mathrm{cop}}\cong P(A^{\mathrm{op}},B^{\mathrm{op}})$.
\end{lem}
\begin{proof}
  First, we show that the monoid $[C^{\mathrm{cop}},B]$ equals
  $[C,B^{\mathrm{op}}]^{\mathrm{op}}$, by showing that the multiplications $\nu$ and $\nu'$ of
  these monoids---which coincide as objects in $\ca{V}$---are equal.
  The multiplication $\nu$ corresponds under the
  tensor--hom adjunction to
  \begin{equation}
    \label{eq:33}
    [C,B]^{\otimes 2}\otimes C
    \xrightarrow{1\otimes1\otimes (c\cdot\Delta)}
    [C,B]^{\otimes 2}\otimes C^{\otimes 2}
    \xrightarrow{1\otimes c\otimes 1}
    ([C,B]\otimes C)^{\otimes 2}
    \xrightarrow{\mathrm{ev}^{\otimes 2}}B^{\otimes2}
    \xrightarrow{\mu} B,
  \end{equation}
  where $c$ denotes the braiding, while $\nu'$ corresponds to
  \begin{equation}
    \label{eq:34}
    [C,B]^{\otimes 2}\otimes C\xrightarrow{c\otimes \Delta}
    [C,B]^{\otimes 2}\otimes C^{\otimes 2}
    \xrightarrow{1\otimes c\otimes1}
    ([C,B]\otimes C)^{\otimes 2}
    \xrightarrow{\mathrm{ev}^{\otimes 2}}
    B^{\otimes2}
    \xrightarrow{\mu\cdot c}B.
  \end{equation}
  Verifying that both composite morphisms are equal provided
  that the braiding $c$ is a symmetry is now routine.

  We complete the proof by exhibiting the following string of natural
  isomorphisms
  \begin{multline}
    \label{eq:32}
    \mathcal{C}(C,P(A,B)^{\mathrm{cop}})\cong
    \mathcal{C}(C^{\mathrm{cop}},P(A,B))\cong
    \mathcal{A}(A,[C^{\mathrm{cop}},B])\cong
    \mathcal{A}(A,[C,B^{\mathrm{op}}]^{\mathrm{op}})\\
    \cong
    \mathcal{A}(A^{\mathrm{op}},[C,B^{\mathrm{op}}])\cong
    \mathcal{C}(C,P(A^{\mathrm{op}},B^{\mathrm{op}})) 
  \end{multline}
  where we abbreviated $\mathcal{C}=\Comon(\mathcal{V})$ and
  $\mathcal{A}=\Mon(\mathcal{V})$.
\end{proof}
\begin{cor}
  \label{cor:6}
  In the situation of Lemma~\ref{l:5}, $P(A,B)$ is a cocommutative comonoid
  provided that $A$ and $B$ are commutative monoids. In particular, $A^\circ$ is
  cocommutative if $A$ is commutative.
\end{cor}
\section{Universal measuring comonoids of cocommutative Hopf monoids}
\label{sec:meas-comon-hopf}

In the classical case of $k$-vector spaces, the finite dual $A^\circ=P(A,k)$ of
a $k$-algebra $A$ is constructed as a subspace of the linear dual $A^*$ (\ref{eq:3}),
and this is used to endow $A^\circ$ with an antipode if $A$ has an antipode $s$. The
argument consists of showing that, if $\alpha\in A^\circ\subset A^*$, the
functional $\alpha\cdot s$ also belongs to $A^\circ$, so the linear map given by
precomposing with the antipode $s$ restricts to $A^\circ$. Exactly the same
argument is carried over to the case of a Noetherian commutative ring $k$
in~\cite{MR1780737}, with the additional hypothesis that $A^\circ$ should be a
pure sub-$k$-module of $k^A$.
In this section we prove that all restrictions on the base
commutative ring $k$ can be lifted, as long as the Hopf algebra $A$ is
cocommutative. More precisely, we prove:

\begin{thm}
  \label{thm:5}
  If $H$ is a cocommutative Hopf monoid in a locally presentable symmetric
  monoidal closed category $\mathcal{V}$, then $P(H,B)$ is a Hopf monoid, for
  any commutative monoid $B$.
\end{thm}
\begin{proof}
  The cocommutativity of $H$ will be used in the fact that the
  comultiplication $\Delta\colon H\to H\otimes H$ is a morphism of
  comonoids. Denote by $s$ the antipode of $H$; it is a monoid morphism
  $H^{\mathrm{op}}\to H$, so $P(s,B)$ is a comonoid morphism $P(H,B)\to
  P(H^{\mathrm{op}},B)\cong P(H,B)^{\mathrm{cop}}$, by Lemma \ref{l:5} and
  commutativity of $B$. We will show that the underlying arrow of $P(s,B)$
  in $\ca{V}$ is an antipode for the bimonoid $P(H,B)$.

  In Remark~\ref{rmk:6} we exhibit the relationship between the opmonoidal
  structure of $[-,B]$ and the monoidal structure of $P(-,B)$,
  via the unit of the adjunction. In
  this proof we shall use the same notations as in the said remark.

  To keep the notation simple, we shall denote all the multiplications by $\mu$
  when no confusion is possible. The multiplication of $P(H,B)$ arises from the
  monoidal structure of $P(-,B)$ and the comonoid structure of $H$. Explicitly,
  it is the composition $P(\Delta,B)\cdot \psi_{H,H}\colon P(H,B)^{\otimes2}\to
  P(H,B)$. It is easy to verify that the corresponding morphism of monoids $H\to
  [P(H,B)^{\otimes 2},B]$ is
  \begin{equation}
    [\psi_{H,H},B]\cdot \eta_{H\otimes H}\cdot\Delta
    =
    \chi_{P(H,B),P(H,B)}\cdot (\eta_H\otimes \eta_H)\cdot \Delta
    \label{eq:38}
  \end{equation}
  where the equality uses one of the diagrams displayed in \eqref{eq:82}.

  Denote the antipode of $H$ by $s$. We are to show the following equality
  \begin{multline}
    \big(
    P(H,B)
    \xrightarrow{\Delta}
    P(H,B)^{\otimes 2}
    \xrightarrow {P(s,B)\otimes 1}
    P(H,B)^{\otimes 2}
    \xrightarrow{\mu}
    P(H,B)
    \big)
    =\\=
    \big(
    P(H,B)
    \xrightarrow{P(\iota,B)}
    P(I,B)\cong I
    \xrightarrow{P(\varepsilon ,B)}
    P(H,B)
    \big)
    ;
    \label{eq:39}
  \end{multline}
  in order to do so, we shall show that the two compositions have equal
  transposes under $[-,B]\dashv P(-,B)$. These transposes are calculated by
  first applying $[-,B]$ and then pre-composing with the unit $\eta_{H}$.

  We have already calculated the transpose of $\mu$ (\ref{eq:38}), from where it follows
  that the transpose of $\mu\cdot(P(s,B)\otimes 1)\cdot\Delta$ is the first
  composition in the following chain of equalities.
  \begin{multline}
    [\Delta,B]\cdot [P(s,B)\otimes 1,B]\cdot \chi_{P(H,B),P(H,B)}\cdot
    (\eta_H\otimes \eta_H)\cdot \Delta_H=\\
    =[\Delta, B]\cdot \chi_{P(H,B),P(H,B)}\cdot([P(s,1),B]\otimes
    1)\cdot(\eta_H\otimes \eta_H)\cdot\Delta_H =\\
    = [\Delta, B]\cdot \chi_{P(H,B),P(H,B)} \cdot (\eta_H\otimes\eta_H)\cdot
    (s\otimes H)\cdot\Delta_H=\\
    =\mu_{[P(H,B),B]}\cdot (\eta_H\otimes\eta_H)\cdot (s\otimes
    H)\cdot\Delta_H=\\
    =\eta_{H}\cdot\mu_H\cdot (s\otimes H)\cdot\Delta_H=
    \eta_H\cdot \iota_H\cdot\varepsilon_H
    \label{eq:40}
  \end{multline}
  The first equality uses the naturality of $\chi$, the second the naturality of
  $\eta$, the third holds since
  $\mu_{[P(H,B),B]}=[\Delta,B]\cdot\chi_{P(H,B),P(H,B)}$ is the convolution
  product of $[P(H,B),B]$; the fourth equality is the fact that $\eta_H$ is a
  monoid morphism, and the last is one of the two antipode axioms.

  On the other hand, the transpose of $P(\varepsilon_H,B)\cdot P(\iota_H,B)$ is
  precisely $\eta_H\cdot\iota_H\cdot\varepsilon _H$, by naturality of
  $\eta$. Therefore, we have proved the equality~\eqref{eq:39}. The other
  antipode axiom for $P(s,B)$ is symmetric to the one just verified, and holds
  by the same argument, concluding the proof.
\end{proof}

\begin{ex}
  \label{ex:12}
  Let $k$ be a commutative ring and $H$ any cocommutative Hopf $k$-algebra. Then
  $P(H,k)=H^\circ$ is a Hopf algebra. In this general case, there is no obvious
  reason why $H^\circ$ should be a sub-$k$-module of $H^*$.
\end{ex}
\begin{ex}
  \label{ex:16}
  This is a good place to examine the meaning of the results so far when the
  base category $\mathcal{V}$ is the category $\mathbf{Set}$ of sets, with its
  monoidal structure given by cartesian product. Each set has a unique
  (cocommutative) comonoid structure (where the multiplication is the diagonal
  function), ie the forgetful functor $\Comon(\mathbf{Set})\to\mathbf{Set}$ is
  an isomorphism.  The universal measuring set $P(A,B)$ of a set a pair of
  monoids $A$ and $B$ is the set $\Mon(\mathbf{Set})(A,B)$ of monoid morphisms
  $A\to B$.

  Any monoid $A$ is automatically a cocommutative bimonoid. If $B$ is a
  commutative monoid, point-wise multiplication endows $\Mon(\mathbf{Set})(A,B)$
  with a monoid structure; compare with Corollary~\ref{dualbimonoid}. A Hopf
  algebra $H$ in $\mathbf{Set}$ is just a group; the antipode $s\colon H\to H$
  is given by $s(x)=x^{-1}$. Theorem~\ref{thm:5}, then, says that
  $\Mon(\mathbf{Set})(H,B)$ is a group if $B$ is commutative and $H$ is a
  group. The inverse of a monoid map is, of course, $(f^{-1})(x)=f(x^{-1})$.
\end{ex}

\section{Universal measuring comonoids of Hopf monoids}
\label{sec:univ-meas-comon}
Having shown that the universal measuring comonoid $P(A,B)$ is a Hopf monoid when
$A$ is a cocommutative Hopf monoid and $B$ is a commutative monoid, we now
investigate the case of a general, not necessarily cocommutative, Hopf monoid
$A$. In order to do so, we need first some basic notions and facts about Hopf
monads and Hopf comonads and Hopf monoids. The main result of
the section, Theorem~\ref{thm:4}, is powerful enough to encompass the examples
of vector spaces and dg vector spaces.

\subsection{Hopf monads}
\label{sec:hopf-monads}

In this section we briefly recall the notion of Hopf monad. More details can be
found in~\cite{MR2793022}. Let $\mathcal{C}$ be a monoidal category and
$\mathsf{T}=(T,\eta,\mu)$ a monad on it. An \emph{opmonoidal structure} on
$\mathsf{T}$ consists of a natural transformation $T_{2,X,Y}\colon T(X\otimes
Y)\to T(X)\otimes T(Y)$ and a morphism $T_0\colon T(I)\to I$ satisfying various
axioms that make the following result of~\cite{MR1887157} hold: the
category $\mathcal{C}^T$ of Eilenberg--Moore algebras has a monoidal structure that makes
the forgetful functor into $\mathcal{C}$ strict monoidal.

We will later be interested in the case of the monad $T=(A\otimes\mathbin{-})$
induced by a bimonoid $A$ in a braided tensor category $\mathcal{C}$. The
opmonoidal structure is given by
\begin{equation}
  \label{eq:68}
  T_{2,X,Y}\colon A\otimes X\otimes Y\xrightarrow{\Delta\otimes 1\otimes 1}
  A\otimes A\otimes X\otimes Y\xrightarrow{1\otimes c_{A,X}\otimes 1}
  A\otimes X\otimes A\otimes Y
\end{equation}
and $T_{0}=\varepsilon\otimes 1\colon A\otimes I\to I\otimes I\cong I$.

Given an opmonoidal monad $\mathsf{T}$ as above, its left and right \emph{fusion
  operators} or \emph{Hopf maps} are the displayed compositions.
\begin{gather}
  \label{eq:69}
  H^\ell_{X,Y}\colon T(X\otimes TY)\xrightarrow{T_{2,X,TY}} TX\otimes T^2Y
  \xrightarrow{1\otimes \mu_Y} TX\otimes TY
  \\
  H^r_{X,Y}\colon T(TX\otimes Y)\xrightarrow{T_{2,TX,Y}} T^2X\otimes
  TY\xrightarrow{\mu_X\otimes1} TX\otimes TY
\end{gather}
The opmonoidal monad $\mathsf{T}$ is \emph{left (resp. right) Hopf} if $H^\ell$
(resp. $H^r$) is invertible.

One of the main results of~\cite{MR2793022} states that, if $\mathcal{C}$ is
left (resp. right) closed, $\mathsf{T}$ is left (resp. right) Hopf if and only
if the monoidal category $\mathcal{C}^T$ is left (resp. right) closed and the
forgetful functor is strong closed (ie it preserves internal homs up to
isomorphism).

\subsection{Hopf comonads}
\label{sec:hopf-comonads}
In the interest of completeness, and since \cite{MR2793022}~gives full
descriptions only for the case of monads, we shall provide some details about
the theory of Hopf comonads. One difference with the case of monads is that,
although there is an abundance of examples of closed categories, even the basic examples of the
category of sets or the category of vector spaces are not coclosed categories
(categories whose opposite categories are closed). In examples,
when tensoring with an object $M$ has a left adjoint, it does so because $M$ has
a dual. Below we briefly treat the relationship between the Hopf condition for
comonads and the existence of duals.

A \emph{monoidal structure} on a comonad $\mathsf{G}=(G,\varepsilon,\delta)$ consists
of natural transformations $G_{2,X,Y}\colon GX\otimes GY\to G(X\otimes Y)$ and a
morphism $G_0\colon I\to G(I)$ satisfying certain axioms that imply that its
category $\mathcal{C}^G$ of Eilenberg--More coalgebras is monoidal and the
forgetful functor $U\colon \mathcal{C}^G\to\mathcal{C}$ is strict monoidal.

Given a monoidal comonad as in the previous paragraph, the right and left fusion
operators are defined in the following way.
\begin{gather}
  H^r_{X,Y}\colon
  G(X)\otimes G(Y)\xrightarrow{1\otimes\delta_Y} G(X)\otimes G^2(Y)
  \xrightarrow{G_{2,X,G(Y)}} G(X\otimes G(Y))
  \label{eq:63}
  \\
  H^\ell_{X,Y}\colon
  G(X)\otimes G(Y)\xrightarrow{\delta_X\otimes1} G^2(X)\otimes G(Y)
  \xrightarrow{G_{2,G(X),Y}} G(G(X)\otimes Y)
  \label{eq:61}
\end{gather}
One says that the comonad $\mathsf{G}$ is \emph{right (resp. left) Hopf} if $H^r$
(resp.~$H^\ell$) is invertible.

Similarly, there are morphisms as displayed below, natural in $\mathsf{G}$-coalgebras
$(M,\chi)$ and $X\in\mathcal{C}$.
\begin{gather}
  \bar H^r_{X,M}\colon
  G(X)\otimes M\xrightarrow{1\otimes\chi} G(X)\otimes G(M)
  \xrightarrow{G_{2,X,M}} G(X\otimes M)
  \label{eq:64}
  \\
  \bar H^\ell_{M,Y}\colon M\otimes G(Y)\xrightarrow{\chi\otimes 1}
  G(M)\otimes G(Y)\xrightarrow{G_{2,M,Y}} G(M\otimes Y)
  \label{eq:65}
\end{gather}
Clearly, $\bar H^r$ is invertible if and only if $H^r$ is invertible; for, each
$\mathsf{G}$-coalgebra is an $U$-split equaliser of cofree coalgebras~\cite[\S VI]{MacLane}.

Let $(M,\chi)$ be a $G$-coalgebra and consider the situation when
$(\mathbin{-}\otimes M)$ has a left adjoint $L$; typically, $L$ is given by
tensoring with a right dual of $M$.
\begin{equation}
  \label{eq:66}
  \xymatrix@C=1.6cm{
    \mathcal{C}^G\ar@<5pt>[r]^-{\mathbin{-}\otimes (M,\chi)}\ar[d]_U
    \ar@<-5pt>@{<..}[r]_{\hat{L}} \ar@{}[r]|-\top&
    \mathcal{C}^G\ar[d]^U\\
    \mathcal{C}\ar@<5pt>[r]^-{\mathbin{-}\otimes M}\ar@{}[r]|-\top
    \ar@<-5pt>@{<-}[r]_-{L}
    &
    \mathcal{C}
  }
\end{equation}
A \emph{lifting} of the adjunction $L\dashv (\mathbin{-}\otimes M)$ to
$\mathcal{C}^G$ is a left adjoint to $\hat L\dashv (\mathbin{-}\otimes(M,\chi))$
that makes $(U,U)$ a strict morphism of adjunctions; this means that the square
formed by the left adjoints commutes and the unit and counit of the respective
adjunctions are compatible with $U$ in an obvious way (see~\cite[\S IV.7]{MacLane}).
\begin{lem}
  \label{l:6}
  The adjunction $L\dashv (\mathbin-\otimes M)$ as above lifts to $\mathcal{C}^G$, for all
  $\mathsf{G}$-coalgebras $(M,\chi)$, if and only if $\mathsf{G}$ is right
  Hopf. Symmetrically, an adjunction $L\dashv(M\otimes\mathbin-)$ lifts to
  $\mathcal{C}^G$, for all $\mathsf{G}$-coalgebras $(M,\chi)$, if and only if
  $\mathsf{G}$ is left Hopf.
\end{lem}
\begin{proof}
  This is dual to part of Theorem~3.6 of~\cite{MR2793022}; in fact, it is dual
  to Theorem~3.13 together with Example~3.12 of op.~cit.
\end{proof}

\begin{rmk}
  \label{rmk:3}
  Let $X$ be an object in the monoidal category $\mathcal{A}$, and suppose given
  an adjunction $(L,(\mathbin{-}\otimes X),\eta,\varepsilon)$. Consider the
  canonical left action of the monoidal category $\mathcal{A}$ on itself, and note
  that the right adjoint is a strong morphism with respect to it, with structure
  given by the associativity and unit constraints
  \begin{equation}
    A\otimes (B\otimes X)\cong(A\otimes B)\otimes X\label{eq:7}
  \end{equation}
  and $I\otimes X\cong X$. By doctrinal
  adjunction~\cite{Doctrinal}, the left adjoint $L$ carries a unique opmorphism structure that
  makes $\eta$ and $\varepsilon$ compatible with the action. Then $X$ has a
  right dual if and only if the opmorphism $L$ is a strong morphism; in which
  case, the right dual is $L(I)$.
\end{rmk}
\begin{lem}
  \label{l:8}
  Suppose given a monoidal comonad on the monoidal category $\mathcal{C}$ and
  adjunctions as in~\eqref{eq:66}, so that $(U,U)$ is a strict morphism of
  adjunctions. Then $(M,\chi)$ has a right dual in $\mathcal{C}^{G}$ provided
  that $M$ has a right dual in $\mathcal{C}$.
\end{lem}
\begin{proof}
  Let $M^\vee$ be the right dual of $M$.
  By Remark~\ref{rmk:3}, the left adjoint $\hat L$ is an opmorphism with respect
  to the left action of $\mathcal{C}^G$ on itself, with structure given by
  morphisms $\lambda\colon \hat L((N,\nu)\otimes(N',\nu'))\to (N,\nu)\otimes\hat
  L(N',\nu')$ whose image under $U$ are the isomorphisms $(N\otimes N')\otimes
  M^\vee\cong N\otimes (N'\otimes M^\vee)$. Thus $\lambda$ is an isomorphism, and $\hat L$
  is isomorphic to $(\mathbin{-}\otimes\hat L(I))$, so $(M,\chi)$ has a right
  adjoint.
\end{proof}

\begin{prop}
  \label{prop:8}
  Let $(M,\chi)$ be a $\mathsf{G}$-coalgebra.
  If $\mathsf{G}$ is right (resp. left) Hopf, then $(M,\chi)$ has a right
  (resp. left) dual in $\mathcal{C}^G$ provided that $M$ have a right
  (resp. left) dual in $\mathcal{C}$.
\end{prop}
\begin{proof}
  Suppose that $M^\vee$ is a right dual to $M$, so $L=(\mathbin{-}\otimes
  M^\vee)$ is a left adjoint to $(\mathbin{-}\otimes M)$. This adjunction
  lifts to an adjunction $\hat L\dashv(\mathbin{-}\otimes (M,\chi))$ on
  $\mathcal{C}^G$ by Lemma~\ref{l:6}. The result can now be deduced from
  Lemma~\ref{l:8}.
\end{proof}
\begin{lem}
  \label{l:12}
  Let $\mathsf{G}$ be a $\kappa$-accessible monoidal comonad on an accessible monoidal
  category $\mathcal{V}$. Then $\mathcal{V}^G$ is an accessible monoidal
  category, and locally presentable if $\mathcal{V}$ is so. Dualizable
  $\mathsf{G}$-coalgebras are $\kappa$-presentable.
\end{lem}
\begin{proof}
  It was mentioned in Corollary~\ref{cor:8} that $\mathcal{V}^G$ is an
  accessible category, with accessible forgetful functor $U$ to $\ca{V}$.
  It remains to be shown
  that the functor $((M,\chi)\otimes-)$ is accessible for any
  $\mathsf{G}$-coalgebra $(M,\chi)$, and similarly tensoring on the other
  side. This is a consequence of \cite[Prop.~2.4.10]{MR1031717}, since $U$ is
  conservative and accessible, and $U((M,\chi)\otimes-)$ is the accessible
  $M\otimes U(-)$. Therefore $\ca{V}^G$ is accessible monoidal as in
  Definition \ref{def:accessible monoidal cat}.

  The hypotheses that $G$ is $\kappa$-accessible and that $\mathcal{V}$ is
  accessible as a monoidal category tell us that $\mathcal{V}$ is
  $\kappa$-accessible and its unit object $I$ is $\kappa$-presentable.
  If $M$ is an object of $\mathcal{V}$ with a left dual, then
  $\mathcal{V}(M,-)\cong\mathcal{V}(I,-\otimes M^\vee)$ is accessible,
  so $M$ is $\kappa$-presentable. If $(M,\chi)$ is a $\mathsf{G}$-coalgebra,
  then $(M,\chi)$ is presentable with presentability degree at least that of $M$
  and that of $G$, ie at least $\kappa$, by Corollary~\ref{cor:8}.
\end{proof}

\begin{prop}
  \label{prop:9}
  Let $\mathsf{G}$ be a $\kappa$-accessible monoidal comonad on the
  accessible monoidal category $\mathcal{V}$. Assume that each
  $\mathsf{G}$-coalgebra is a $\kappa$-filtered colimit of right (resp. left)
  dualizable $\mathsf{G}$-coalgebras. Then $\mathsf{G}$ is right (resp. left) Hopf if and
  only if each right (resp. left) dualizable $\mathsf{G}$-coalgebra has a right
  (resp. left) dual in $\mathcal{C}^G$.
\end{prop}
\begin{proof}
  We briefly deal with the case of right dualizable $\mathsf{G}$-coalgebras. The direct
  implication holds by Proposition~\ref{prop:8}. For the converse, dualizable
  objects of $\mathcal{V}^G$ are $\kappa$-presentable, by Lemma~\ref{l:12}.
  By the $\kappa$-accessibility
  of $G$, the domain and codomain of $\bar H^r$ (see~\eqref{eq:64}) preserve
  $\kappa$-filtered colimits, so $\bar H^r$ is an isomorphism if it is so on
  dualizable objects.
\end{proof}

\subsection{Hopf (co)monads induced by Hopf monoids}
\label{sec:hopf-monoids}

Let $\mathcal{C}$ be a braided monoidal category and
$(H,\Delta,\varepsilon,\mu,\iota)$ bimonoid in it. There are eight questions that
we may naturally ask ourselves about $H$.
\begin{quotation}
  Is the monad $(H\otimes\mathbin-)$ left or right Hopf? Is the monad
  $(\mathbin-\otimes H)$ left or right Hopf? Is the comonad
  $(H\otimes\mathbin-)$ left or right Hopf? Is the comonad $(\mathbin-\otimes
  H)$ left or right Hopf?
\end{quotation}
In this section, we give
the answers to the above questions in terms of the Hopf maps of $H$.
There are four such operators --~see
sections \ref{sec:hopf-monads} and \ref{sec:hopf-comonads}~-- depicted in
Figure~\ref{fig:2}. We call the first two maps \emph{fusion operators} and the
last two \emph{opfusion operators}. The contents of this section derive
from~\cite[Lemma~5.1]{MR2793022}; our contribution consists in the inclusion of
the dual statements --~ie those for right Hopf monads~-- and a condensed proof.

\begin{prop}\label{prop:10}
\begin{enumerate}[leftmargin=*]
  \item \label{item:25} The following are equivalent for a bimonoid $H$.
    \begin{enumerate}
    \item \label{item:11} $H$ is a Hopf monoid.
    \item \label{item:12} The two fusion operators of Figure~\ref{fig:2} are
      invertible.
    \item \label{item:34} Any one of the two fusion operators in
      Figure~\ref{fig:2} is invertible.
    \item \label{item:17} The opmonoidal monad $(H\otimes\mathbin-)$ is left
      Hopf.
    \item \label{item:18} The opmonoidal monad $(\mathbin-\otimes H)$ is right
      Hopf.
    \item \label{item:23} The monoidal comonad $(H\otimes\mathbin-)$ is right
      Hopf.
    \item \label{item:24} The monoidal comonad $(\mathbin-\otimes H)$ is left
      Hopf.
    \end{enumerate}
  \item \label{item:26} The following are equivalent for a bimonoid $H$.
    \begin{enumerate}[resume]
    \item \label{item:27} $H$ is an op-Hopf monoid.
    \item \label{item:28} The two opfusion operators of Figure~\ref{fig:2} are
      invertible.
    \item \label{item:35} Any one of the two opfusion operators in
      Figure~\ref{fig:2} is invertible.
    \item \label{item:29} The opmonoidal monad $(H\otimes\mathbin-)$ is right
      Hopf.
    \item \label{item:30} The opmonoidal monad $(\mathbin-\otimes H)$ is left
      Hopf.
    \item \label{item:31} The monoidal comonad $(H\otimes\mathbin-)$ is left
      Hopf.
    \item \label{item:32} The monoidal comonad $(-\otimes H)$ is right
      Hopf.
    \end{enumerate}
  \end{enumerate}
\end{prop}
\begin{proof}
  We prove the equivalences of the statements in
  (\ref{item:25}), leaving the proof of the equivalences in (\ref{item:26}) for the reader. Consider the function
  \begin{equation}
    \label{eq:27}
    \Phi\colon
    \mathcal{V}(H,H)\longrightarrow \mathcal{V}(H^{\otimes2},H^{\otimes2}),\quad
    f\mapsto (H\otimes\mu)\cdot (H\otimes f\otimes H)\cdot (\Delta\otimes H).
  \end{equation}
  It is easy to verify that $\Phi$ is a monoid morphism if the domain has the
  convolution product induced by the bimonoid $H$, and the codomain has the
  product given by composition. In fact, $\Phi$ is an isomorphism onto the
  monoid consisting of those endomorphisms that are simultaneously: endomorphisms
  of left $H$-comodules on the cofree left $H$-comodule $H^{\otimes 2}$; and,
  endomorphisms of free right $H$-modules on the free left $H$-module
  $H^{\otimes2}$. From these considerations it follows that $1_H$ has a
  convolution inverse, ie there exists an antipode, if and only if $\Phi(1_H)$,
  the first fusion operator in Figure~\ref{fig:2}, is invertible.

  Now consider the function
  \begin{equation}
    \label{eq:67}
    \Phi'\colon
    \mathcal{V}(H,H)\longrightarrow \mathcal{V}(H^{\otimes2},H^{\otimes2}),\quad
    f\mapsto (\mu\otimes H)\cdot(H\otimes f\otimes H)\cdot(H\otimes \Delta).
  \end{equation}
  Again, it is easy to verify that $\Phi'$ is an anti-morphism of monoids when
  the domain is equipped with convolution and the codomain with
  composition. Furthermore, it is an isomorphism onto the submonoid of those
  endomorphisms of $H^{\otimes2}$ that are simultaneously: right $H$-comodule
  endomorphisms on the cofree $H$-comodule $H^{\otimes2}$; and, left $H$-module
  endomorphisms on the free left $H$-module $H^{\otimes2}$. Therefore, there
  exists an antipode for $H$ if and only if $\Phi'(1_H)$, which is the second
  fusion operator in Figure~\ref{fig:2}, is invertible. These first two
  paragraphs of the proof show the equivalence between the conditions
  (\ref{item:11}), (\ref{item:12}) and~(\ref{item:34}).

  The left fusion operator of the opmonoidal monad $(H\otimes-)$, the right
  fusion operator of the opmonoidal monad $(-\otimes H)$, the right fusion
  operator of the monoidal comonad $(H\otimes-)$ and the left fusion operator of
  the monoidal comonad $(-\otimes H)$, have components
  depicted in the respective order in Figure~\ref{fig:4}.
  \begin{figure}
  \begin{gather}
    HXHY\xrightarrow{\Delta XHY}HHXHY\xrightarrow{Hc_{H,X}HY}HXHHY
    \xrightarrow{HX\mu Y} HXHY
    \\
    XHYH\xrightarrow{XHY\Delta}XHYHH \xrightarrow{XHc_{YH}H}XHHYH
    \xrightarrow{X\mu YH} XHYH
    \\
    HXHY\xrightarrow{HX\Delta Y}HXHHY\xrightarrow{Hc_{X,H}HY} HHXHY
    \xrightarrow{\mu XHY}HXHY
    \\
    XHYH\xrightarrow{X\Delta YH} XHHYH\xrightarrow{XHc_{H,Y}H}XHYHH
    \xrightarrow{XHY\mu}XHYH
  \end{gather}
  \caption{Fusion operators of the monads $(H\otimes-)$ and $(-\otimes H)$, and
    of the comonads $(H\otimes-)$ and $(-\otimes H)$.}\label{fig:4}
  \end{figure}
  Setting $X=Y=I$ in these fusion operators we obtain, respectively, $h$, $h'$, $h'$ and $h$. In fact, the
  general components can be easily obtained from $h$ and $h'$ by tensoring with
  $X$ and $Y$ and composing with the braiding; for example, the first composite is
  obtained as
  \begin{equation}
    (H\otimes c_{H,X}\otimes Y)\cdot (h\otimes X\otimes Y)\cdot
    (H\otimes c_{X,H}^{-1}\otimes Y)
    \label{eq:74}
  \end{equation}
  Therefore, any one of the four fusion operators depicted above is invertible if
  and only if either $h$ or $h'$ is invertible, completing the proof
  of~\ref{item:25}.
\end{proof}

\begin{figure}
  \begin{gather}
    h 
    \colon H^{\otimes 2}\xrightarrow{\Delta\otimes 1} H^{\otimes3}
    \xrightarrow{1\otimes\mu} H^{\otimes2}
    \\
    h'
    \colon H^{\otimes2}\xrightarrow{1\otimes \Delta} H^{\otimes3}
    \xrightarrow{\mu\otimes1} H^{\otimes2}
    \\
    \bar h 
    \colon H^{\otimes 2}\xrightarrow{\Delta\otimes1} H^{\otimes
      3}\xrightarrow{1\otimes c_{H,H}}H^{\otimes 3} \xrightarrow{\mu\otimes1}
    H^{\otimes 2}
    \\
    \bar h' 
    \colon H^{\otimes2} \xrightarrow{1\otimes \Delta} H^{\otimes3}
    \xrightarrow{c_{H,H}\otimes 1} H^{\otimes3}\xrightarrow{1\otimes\mu}
    H^{\otimes2}
  \end{gather}
  \caption{The two fusion operators (top) and the two opfusion operators (bottom).}
\label{fig:2}
\end{figure}

\subsection{Universal measuring comonoids and Hopf monoids}
\label{sec:univ-meas-coalg}

Recall from Section \ref{Kleislicats} the Kleisli category for a monoidal monad on
a monoidal category, and from Definition \ref{df:3} the category
$\avb$, for a pair of monoids $A$, $B$: it is the pullback (\ref{eq:24}) of the forgetful
$(\mathcal{V}_B)^{(A\otimes-)}\to\mathcal{V}_B$ along the universal Kleisli
functor $F_B\colon\mathcal{V}\to\mathcal{V}_B$. Here $\mathcal{V}_B$ is the
Kleisli category of the monad $(-\otimes B)$. The monad $(-\otimes B)$ is monoidal
if the monoid $B$ is commutative. Furthermore, it is an easy calculation to verify
that if $\mathcal{V}$ is a \emph{symmetric} monoidal category, then
$(-\otimes B)$ is a braided monoidal functor when $B$ is commutative, see
Lemma \ref{Hsymmetriclax}; then
$\mathcal{V}_B$ is a symmetric monoidal category.

\begin{prop}
  \label{prop:4}
  Let $\mathcal{V}$ be a symmetric monoidal closed category and $B$ a
  commutative monoid in it.
  Suppose that $A$ is a Hopf (resp. op-Hopf) monoid in $\mathcal{V}$. Then $\avb$
  is monoidal and an object $(X,\psi)\in\avb$ has a left (resp. right) dual if
  and only if $X$ has a dual in $\mathcal{V}$.
\end{prop}
\begin{proof}
  The category $\mathcal{V}_B$ is symmetric monoidal and
  the Kleisli functor $F_B\colon\mathcal{V}\to\mathcal{V}_B$ is
  braided and strict monoidal; see Section \ref{Kleislicats}. Thus,
  $A=F_B(A)$ is a Hopf (resp. op-Hopf) monoid in
  $\mathcal{V}_B$, so the category  $(\mathcal{V}_B)^{(A\otimes\mathbin{-})}$ of
  left $A$-modules in $\mathcal{V}_B$ is monoidal and the forgetful functor
  $(\mathcal{V}_B)^{(A\otimes\mathbin{-})}\to\mathcal{V}_B$ strict
  monoidal. Then, the pullback $\avb$ defined in~\eqref{eq:24} is
  a monoidal category and the forgetful $\avb\to\mathcal{V}$ is strict
  monoidal. Suppose that an object $(X,\psi)$ of $\avb$ is dualizable, ie $X$
  has a dual in $\mathcal{V}$. Then $F_B(X)$ has a dual in $\mathcal{V}_B$, by
  the strict monoidality of $F_B$, so the projection $(X,\varphi)$ of $(X,\psi)$ to
  $(\mathcal{V}_B)^{(A\otimes-)}$ is dualizable. By the hypothesis of being
  Hopf (resp. op-Hopf), we have that $(X,\varphi)$ has a left (resp. right) dual and the
  forgetful functor into $\mathcal{V}_B$ preserves evaluation and
  coevaluation. It follows that $(X,\psi)\in\avb$ has a left (resp. right) dual
  by the definition of $\avb$ as a pullback.
\end{proof}

\begin{df}
  A braided monoidal category $\mathcal{V}$ is said to satisfy the
  \emph{fundamental theorem of comodules} if, for each comonoid $C$ in
  $\mathcal{V}$, each $C$-comodule is a filtered colimit of dualizable
  $C$-comodules.\label{df:6}
\end{df}
\begin{rmk}
  The previous definition elicits a number of comments. First, it seems possible
  to drop the assumption that the monoidal category be braided, distinguishing
  between left and right dualizable objects. We prefer to keep the definition
  more readable by retaining the braiding assumption.

  Secondly, \cite{McCrudden:Maschkean}~says that $\mathcal{V}$ satisfies the
  fundamental theorem of comodules if, for each comonoid $C$, each $C$-comodule
  is filtered colimit of dualizable strong subobjects. We do not require the
  colimit to be one of subobjects.

  Thirdly, one might think that there is a certain ambiguity in our definition with respect
  to left and right $C$-comodules. It is not a real one, however, since left
  $C$-comodules are right comodules over the opposite comonoid.
\label{rmk:5}
\end{rmk}


\begin{thm}
  \label{thm:4}
  Let $\mathcal{V}$ be a locally presentable symmetric monoidal closed category
  that satisfies the fundamental theorem of comodules, $A$ a Hopf (resp. op-Hopf)
  monoid and $B$ a commutative monoid. Then
  $P(A,B)$ is a Hopf (resp. op-Hopf) monoid.
\end{thm}
\begin{proof}
  The category of dualizable right $P(A,B)$-comodules is monoidally isomor\-phic over
  $\mathcal{V}$ to the category of dualizable objects in $\avb$, by
  Corollary~\ref{cor:3} and Corollary~\ref{cor:1}, and the latter category is left (resp. right) autonomous by
  Proposition~\ref{prop:4}. So, any dualizable comodule has a left
  (resp. right) dual. By the fundamental theorem of comodules,
  each $P(A,B)$-comodule is a colimit of dualizable ones, so the comonad
  $(-\otimes P(A,B))$ is left (resp. right) Hopf by
  Proposition~\ref{prop:9}. This is equivalent to saying that $P(A,B)$ is a Hopf
  (resp. op-Hopf) monoid, by Proposition~\ref{prop:10}.
\end{proof}

\begin{ex}
  \label{ex:11}
    If $A$ is a Hopf algebra over a field $k$, then $P(A,B)$
    is a Hopf algebra for any commutative $k$-algebra
    $B$.
     Let $A^{\mathrm{op}}$ be the bialgebra obtained by taking
    the opposite multiplication but leaving the comultiplication intact. If
    $A^{\mathrm{op}}$ is a Hopf algebra, then $P(A,B)$ is op-Hopf.
    The example of graded (co)algebras is explored in the next section.
\end{ex}

\section{Example: graded (co)algebras}
\label{sec:exampl-dg-coalg}

  Recall from Example~\ref{ex:14} that the category $\mathbf{gVect}_{\mathbb{Z}}$ of $\mathbb{Z}$-graded
  vector spaces is a locally finitely presentable
  symmetric monoidal closed category. In what follows, graded (co)algebra
  (Example \ref{ex:5}) and graded (co)module mean (co)monoid and (co)module
  in the said monoidal category.

\begin{lem}
  \label{l:3}
  Let $M$ be a (right) graded comodule over a graded coalgebra $C$. Any homogeneous
  finite-dimensional space of $M$ is contained in a finite-dimensional
  sub graded comodule. 
\end{lem}
\begin{proof}
  The proof is identical to that of~\cite[Lemma~1.1]{Getzler-Goerss}, except
  that we admit negative grading.
\end{proof}

The above lemma immediately yields:
\begin{cor}
  \label{cor:9}
  The category $\mathbf{gVect}_{\mathbb{Z}}$ of graded
  vector spaces satisfies the fundamental theorem of comodules.
\end{cor}
\begin{proof}
  The category of graded vector spaces has an internal hom given by
  \begin{equation}
    \label{eq:8}
    \operatorname{Hom}(X,Y)_n=\prod_{i\in \mathbb{Z}}[X_i,Y_{i+n}]
  \end{equation}
  and unit object $I=k[0]$ the base field $k$ concentrated on degree $0$, as mentioned
  in Example \ref{ex:14}. We
  have to show that finite-dimensional graded spaces have a dual object in
  $\mathbf{gVect}_{\mathbb{Z}}$; for this suppose that $X$ is finite-dimensional,
  ie it is $0$ except in finitely many degrees, say between $-n$ and $n$. If $X$
  had a dual, it should be $\operatorname{Hom}(X,k[0])$
  \begin{equation}
    \label{eq:10}
    \operatorname{Hom}(X,k[0])_m=[X_{-m},k].
  \end{equation}
  By general considerations on duals and internal homs, $X$ has a dual if and
  only if the comparison morphism $\operatorname{Hom}(X,k[0])\otimes
  Y\to\operatorname{Hom}(X,Y)$, with $m$-component
  \begin{equation}
    \label{eq:9}
    \sum_{i=-n}^n [X_{-i},k]\otimes Y_{m-i}\longrightarrow \prod_{j={-n}}^n[X_j,Y_{j+m}]
  \end{equation}
  is an isomorphism. Since the product in the codomain is finite, it can be
  replaced by a sum, and reindexing, it is isomorphic to
  $\sum_{\ell=-n}^n[X_{-\ell},Y_{m-\ell}]$. It is now easy to see
  that~\eqref{eq:9} is the sum of the isomorphisms of the type $[V,k]\otimes
  W\cong [V,W]$ for $V$ a finite-dimensional vector space.
\end{proof}
\begin{rmk}
  \label{rmk:1}
  In the above corollary, it was important that the grading is over the group
  of integers. For example, in the category of non-negatively ($\mathbb{N}$-) graded spaces,
  very few objects have a dual: they are all
  concentrated in degree $0$.
\end{rmk}

We can now describe the result of applying Theorem~\ref{thm:4} to the base
category $\mathcal{V}=\mathbf{gVect}_{\mathbb{Z}}$.
\begin{prop}
  Let $H$ be a $\mathbb{Z}$-graded bialgebra and $B$ a commutative
  $\mathbb{Z}$-graded algebra. Then $P(H,B)$ is a $\mathbb{Z}$-graded
  bialgebra. If $H$ is a Hopf (resp. op-Hopf) graded algebra, then $P(H,B)$ is
  Hopf (resp. op-Hopf) too.\label{prop:1}
\end{prop}

If, instead of $\mathbb{Z}$-graded spaces, we wanted to work with
$\mathbb{N}$-graded spaces, and obstacle presents itself:
$\mathbf{gVect}_{\mathbb{N}}$ does not have enough objects with duals to satisfy
the fundamental theorem of comodules --~Definition~\ref{df:6}. We can say
something, however, if we admit the restriction to (graded) (co)commutative
algebras.
\begin{prop}
  \label{prop:3}
  Let $H$ be a cocommutative $\mathbb{N}$-graded Hopf algebra and $B$ a commutative
  $\mathbb{N}$-graded algebra. Then $P(H,B)$ is a $\mathbb{N}$-graded Hopf algebra.
\end{prop}
The proof of the proposition is an application of Theorem~\ref{thm:5}.

An $\mathbb{N}$-graded vector space is \emph{connected} if its component of
degree $0$ is one-dimensio\-nal. It is well known that connected
$\mathbb{N}$-graded bialgebras automatically are Hopf algebras. Even if one is
only interested in connected spaces, Proposition~\ref{prop:3} is not redundant,
as $P(H,B)$ may not be connected even when $H$ and $B$ are so. For example, a
morphism of $\mathbb{N}$-graded coalgebras $k[0]\to C$ is equivalently given by
an element $g\in C_0$ that is a group-like element of $C$, ie
$\Delta(g)=g\otimes g$ and $\varepsilon(g)=1$. If $C$ is connected, there is at
most one such element, as the restriction of $\varepsilon\colon C\to k[0]$ to
degree $0$ is an isomorphism $C_0\cong k$. Therefore, if $P(A,B)$ is connected,
then there exists at most one $\mathbb{N}$-graded morphism of algebras
$A\to B$, by the definition of $P(A,B)$. An example where this does not happen,
and therefore where $P(A,B)$ is not connected, is that of $k=\mathbb{F}_2$, the
field of characteristic $2$ (so $-1=1$ and graded (co)commutativity is just
ordinary (co)commutativity), and $A=\mathbb{F}_2[x]$ is the polynomial algebra
with the usual Hopf algebra structure, whose cocomultiplication is given by
$\Delta(x)=1\otimes x+x\otimes 1$. For any connected $\mathbb{F}_2$-algebra
$B$, a morphism of graded algebras $\mathbb{F}_2[x]\to B$ is defined by a unique
element of $B_1$. In this way, any $B$ for which $B_1\neq 0$ provides an example
in which $P(\mathbb{F}_2[x],B)$ is not connected.

\bibliographystyle{abbrv}
\bibliography{myreferences}

\end{document}